\documentclass[11pt,twoside,reqno]{amsart}
\usepackage{amsmath, amsfonts, amsthm, amssymb,amscd, graphicx, amscd}
\usepackage{float,epsf}
\usepackage[english]{babel}
\usepackage{enumerate}
\usepackage{tikz}
\usepackage{mathrsfs}
\usepackage[compress,sort]{cite}
\usepackage{srcltx}
\usepackage{geometry}
\usepackage{verbatim}
\usepackage{hyperref}
\usepackage{enumitem} 
\usepackage{bbm} 
\usepackage{relsize}
\usepackage{exscale}
\usepackage[mathscr]{eucal}
\usepackage{refcount}
\usepackage[T1]{fontenc}
\usepackage{wrapfig}
\usepackage{mathtools}
\usepackage{tikz-cd}
\usepackage{blindtext}
\usepackage{soul}
\usepackage{bm}
\usepackage{xcolor}
\usepackage{framed}

\newtheorem{theorem}{Theorem}[section]

\newtheorem{definition}[theorem]{Definition}

\newtheorem{proposition}[theorem]{Proposition}

\theoremstyle{definition}

\newtheorem{example}[theorem]{Example}

\newtheorem*{continuationx}{\continuationtext}

\makeatletter
\newenvironment{continuation}[1]%
  {\phantomsection%
   \newcommand{\continuationtext}{Example \getrefnumber{#1} (continued)}%
   \protected@edef\@currentlabel{\getrefnumber{#1}}
   \begin{continuationx}}%
  {\end{continuationx}}
\makeatother

\numberwithin{equation}{section}
\numberwithin{figure}{section}

\newcommand{\res}{\mathop{\hbox{\vrule height 7pt width .5pt depth 0pt
\vrule height .5pt width 6pt depth 0pt}}\nolimits}

\setlist[enumerate,1]{label=\arabic*.}

\def\rightangle{\vcenter{\hsize5.5pt
    \hbox to5.5pt{\vrule height7pt\hfill}
    \hrule}}
\def\rtangle{\mathrel{\rightangle}}
\def\intave#1{\int_{#1}\hbox{\llap{$\raise2.3pt\hbox{\vrule
height.9pt width7pt}\phantom{\scriptstyle{#1}}\mkern-2mu$}}}

\makeatletter
\def\@citestyle{\m@th\upshape\mdseries}
\let\citeform\@firstofone
\def\@cite#1#2{{%
  \@citestyle[\citeform{#1}\if@tempswa, #2\fi]}}
\makeatother

\hypersetup{bookmarksdepth=2}

\newcommand{\N}{\mathbb{N}}

\newcommand{\R}{\mathbb{R}}

\newcommand{\h}{{\mathcal{H}}}
\newcommand{\lb}{{\mathcal{L}}}
\newcommand{\rb}{{\mathcal{\partial^*}}}
\newcommand{\F}{\bm{F}}
\newcommand{\mF}{\bm{\mathscr{F}}}
\newcommand{\n}{\bm{\nu}}

\DeclareMathOperator{\Div}{div}

\DeclareMathOperator{\Curl}{curl}

\DeclareMathOperator{\spt}{spt}
\DeclareMathOperator{\Lip}{Lip}
\DeclareMathOperator{\wk}{\overset{\ast}{\rightharpoonup}}
\DeclareMathOperator{\dif}{\mathrm{d}\!}

\colorlet{shadecolor}{blue!20}

\begin{document}
\title[On the Tangential Traces of Curl-Measure Fields]{On the Tangential Traces of Curl-Measure Fields}

\author{Cian Nolan}
\address{Cian Nolan,
Department of Mathematics, Purdue University,
 West Lafayette, IN 47907-2067, USA}
 \email{nolan70@purdue.edu}

\author{Monica Torres}
\address{Monica Torres,
Department of Mathematics, Purdue University,
 West Lafayette, IN 47907-2067, USA}
 \email{torresm@purdue.edu}


\begin{abstract}
Curl-measure fields are $p$-integrable vector fields whose distributional curl is a vector-valued Radon measure with finite total variation. They were introduced in \cite{CurlMeasure25}, where, for $p= \infty$, the existence of tangential traces for bounded Lipschitz domains was established, together with the tangential property of the trace. In this paper, we show that the same tangential property holds for domains that are sets of finite perimeter.
\end{abstract}
\maketitle

\bigskip
\bigskip
\section{Introduction}
{\it Trace theorems} are used to obtain generalized integration by parts formulae. These formulae relate integrals over objects like surfaces or solids to integrals over their boundary. 
An example of such a theorem is that if we consider a smooth vector field $\F:\R^3\to\R^3$ and $V\subset\R^3$ a bounded open set with smooth boundary $\partial V$, then for any $\varphi\in C_c^\infty(\R^3)$ we have a {\it Gauss-Green formula},
\begin{equation}\label{mainformula2}
    \int_{\partial V}\varphi\F\times\n_{\partial V} \dif \h^2 = \int_{V}\varphi\Curl\bm{F}\dif x-\int_{V}\bm{F}\times\nabla\varphi \dif x,
\end{equation}
where $\n_{\partial V}$ is the inner normal to $V$. This follows from the fact that 
\begin{equation}\label{curl div theorem}
    \int_V\Curl\F\dif x = \int_{\partial V}\F\times\n_{\partial V}\dif\h^2
\end{equation}
which, in turn follows from the classical divergence theorem.

Formula \eqref{mainformula2} can be established under  weaker hypotheses on the vector field $\F$ and the domain $V$. In this paper, for an open set $\Omega\subset\R^3$, we consider the case where $\F \in \mathcal{CM}^{p}(\Omega)$, where  $\mathcal{CM}^p(\Omega)$ denotes the class of all vector fields $\F\in L^p(\Omega;\R^3)$ whose distributional curl is a vector-valued measure with finite total variation. These objects are called {\it curl-measure fields}. Our region $V$ will be a bounded set of finite perimeter, $E \Subset \Omega$. The main challenge in this problem is that, while for sets of finite perimeter there are well-known notions of {\it reduced boundary} and {\it measure-theoretic inner normal} (denoted by $\partial^* E$ and $\n_E$, respectively), the restriction of $\F\times\n_{E}$ to $\partial^* E$  may not be well defined. This necessitates the introduction of a {\it tangential trace}, which will play the role of generalizing $\F\times\n_{\partial V}$ in \eqref{mainformula2}. We will show that this trace is an $L^\infty$ function on $\rb E$ when $p=\infty$.

It is easily seen in the smooth setting that the term {\it tangential trace} reflects the geometric property of this vector-valued function: At every point $y \in \partial V$, $\F(y)\times\n_{\partial V}(y)$ is tangential to $\partial V$. This tangential property was established for essentially bounded curl-measure fields and Lipschitz domains $V$ in \cite{CurlMeasure25}. In this paper, we show that this tangential property also holds for sets of finite perimeter.

Curl-measure fields are used to establish a framework for problems involving singular vorticity. In particular, the derivation of the Birkhoff-Rott equation (as in \cite{Caflisch}) which governs the evolution of a {\it vortex sheet}, can be considered in a more general setting. The theory of this paper will be connected to this motivating example throughout, however, more detail and a wealth of other related problems in continuum and fluid mechanics may be found in \cite{CurlMeasure25}.

Similarly defined objects to curl-measure fields are the {\it extended divergence-measure-fields} which are finite vector-valued Radon measures $\F$ whose divergence is also a finite signed Radon measure. The space of such fields is denoted as $\mathcal{DM}^{\textnormal{ext}}(\Omega)$, while $\mathcal{DM}^{p}(\Omega) \subset \mathcal{DM}^{\textnormal{ext}}(\Omega)$ is the subspace of $L^{p}$-divergence-measure fields. Their connection to hyperbolic conservation laws was first noted in \cite{ChenFrid99}. The theory of divergence-measure fields is well-developed in the three main cases: $\F \in L^{\infty}$ in \cite{GG09, One-sided-approx, ChenTorres05, RoughOpen}, $\F \in L^{p}$, $1 \leq p < \infty$ in \cite{ChenComiTorres19}, and $\F$ a vector-valued measure in  \cite{ExtDiv25}.  This framework allows the study of initial boundary value problems and the structure and regularity of entropy solutions of hyperbolic systems of conservation laws (e.g., \cite{crippa2024}, \cite{ChenRascle}, \cite{ChenTorres11}, \cite{Vas}). See also \cite{ anzellotti1983traces,ACM,CT21, ComiExtended,ComiMagnani,ComiPayne20,Crasta2019Anzellotti,CrastaDeCicco2, SchonherrSchuricht,Sch07,Silhavy05,Silhavy08,Silhavy09} 
and the references therein for further developments of the theory of divergence-measure fields. Divergence-measure fields also appear in many other areas of analysis, including the prescribed mean curvature equations, the 1-Laplacian, the continuity equation, and related topics (e.g.,  \cite{KS, SS1,SS2,SS3, Leo1, Leo2, Leo3}). More general differential operators than curl and divergence are considered in \cite{BoundedA-Variation}, where the authors also investigate the geometric features of traces. However, their analysis does not extend to domains which are sets of finite perimeter. We focus on these properties here, as they are particularly relevant for applications.

A fundamental difference between divergence-measure fields and curl-measure fields is that the corresponding traces are {\it normal traces} (scalar-valued functions) and  {\it tangential traces} (vector-valued functions), respectively. The first type of trace appears, for example, in the study of {\it shock waves} where there is a jump of the normal trace. The second type of trace appear in the analysis of {\it vortex sheets} where there is a jump between the tangential trace. As seen in \cite{CurlMeasure25}, for Lipschitz domains, the tangential traces possess the important {\it tangential property}, that is, they are tangential to the boundary almost everywhere.

The main purpose of this paper is to establish the tangential property of the trace of curl-measure fields on bounded sets of finite perimeter. This more general case includes that studied in \cite{CurlMeasure25}, but since we are no longer dealing with Lipschitz boundaries (which the proof in \cite{CurlMeasure25} relies upon), we must consider a different approach. Our method is motivated by \cite{ChenTorres05} and \cite{Silhavy05}, where the Gauss-Green formula for essentially bounded divergence-measure fields on sets of finite perimeter was first established via different methods. Upon careful analysis of the two different approaches, and the corresponding adaptations of the techniques to curl-measure fields, we establish the tangential property using the weak* convergence of smooth traces to the tangential trace. This is done together with a representation of the measure $\Curl \F$ in terms of the interior and exterior tangential traces.

The structure of this paper is as follows: we introduce preliminary material in Section \ref{prelims}. Curl-measure fields and tangential traces are defined and discussed in Section \ref{cm fields section}. We devote Section \ref{product method section} to using the approach of \cite{ChenTorres05} to establish a trace theorem for curl-measure fields on sets of finite perimeter. Then in Section \ref{functional method section}, we use the approach of \cite{Silhavy05} to establish a pointwise defined formula for the tangential traces. Finally, in Section \ref{tangential property section}, we bring objects from these two methods together and establish the tangential property in our main Theorem \ref{maintheorem}. 

\section{Preliminaries}\label{prelims}
Throughout, $\Omega$ is assumed to be some open subset of $\R^n$. We will specify later that $n=3$. $\Lip_c(\Omega)$ denotes the set of Lipschitz functions on $\Omega$ with compact support. Unless otherwise specified, the results noted in this section may be found in Chapters 4, 12 and 15 of \cite{Maggi} and Chapters 1,2 and 3 of \cite{AmbrosioFuscoPallara}.
\subsection{Measures}\label{measures}
A {\it Radon measure} on $\Omega$ is an outer measure $\mu$ on $\Omega$ that is locally finite and Borel regular. By this we mean that $\mu(K)<\infty$, for all compact $K\subset\Omega$, the Borel subsets of $\Omega$ are measurable, that is $\mathcal{B}(\Omega)\subset\mathcal{M}(\mu)$, and for every $F\subset\Omega$, there exists $E\in\mathcal{B}(\Omega)$ such that $F\subset E$ and $\mu(F)=\mu(E)$. $\lb^n$ and $\h^n$ will denote the $n$-dimensional Lebesgue and Hausdorff measures on $\R^n$, respectively. If $E$ is a $\lb^n$-measurable set, we may also denote $\lb^n(E)$ by $|E|$.

Let $\mathcal{M}$ be $\sigma$-algebra of subsets of $\Omega$. A {\it $\R^m$-valued measure} on $\Omega$ is a map $\mu:\mathcal{M}\to\R^m$ that is additive on countable collections of disjoint subsets of $\mathcal{M}$.

Consider a set function defined on the bounded Borel subsets of $\Omega$, $\mu:\mathcal{B}_b(\Omega)\to\R^m$. If, for all compact $K\subset\Omega$, $\mu|_{\mathcal{B}(K)}$ is a $\R^m$-valued measure on $\Omega$ with respect to the $\sigma$-algebra $\mathcal{B}(K)$, then we say that $\mu$ is a {\it$\R^m$-valued Radon measure on $\Omega$} (i.e., a member of $\mathcal{M}_{\text{loc}}(\Omega,\R^m)$).

For a Borel set $A\subset\Omega$ and $p\in[1,\infty]$, we will use the convention that $f\in L^p(A;\R^m,\mu)$ if $f$ is a $\mu$-measurable function from $A$ to $\R^m$ and $||f||_{L^p(A)}<\infty$, where
\begin{align*}
    ||f||_{L^p(A)}\coloneqq
    \begin{cases}
        \left(\int_A|f|^p\dif\mu\right)^{1/p}, \qquad &p\in[1,\infty),\\
        \inf\{M: |f|\leq M\, \mu\text{-a.e.}\}, \qquad &p=\infty.
    \end{cases}
\end{align*}
When the set $A$ being considered is clear, we may write $||\cdot||_{L^p(A)}$ as $||\cdot||_{p}$. If $m=1$, we may simply say $f\in L^p(A,\mu)$. We follow similar conventions with $L^p_{\text{loc}}(A;\R^m,\mu)$ which is the space of $\mu$-measurable functions $f:A\to\R^m$, such that for any compact $K\subset A$, $f|_K\in L^p(A;\R^m,\mu)$.

Using the Riesz–Markov–Kakutani representation theorem, we can see that there is a 1-1 correspondence between the bounded linear functionals on $C_c(\Omega,\R^m)$ and the $\R^m$-valued Radon measures on $\Omega$. It gives us a {\it polar decomposition} of these objects into the product of a function and the {\it total variation} of the object, denoted by $|\cdot|$. In particular, we will use that if $\mu$ is Radon measure on $\Omega$ and $f\in L^1_{\text{loc}}(\Omega;\R^m,\mu)$, then we may define a bounded linear functional $f\mu:C_c(\Omega;\R^m)\to\R$ by setting
\[
f\mu(\varphi)=\int_{\Omega}(\varphi\cdot f)\dif\mu, \qquad\text{ for all } \varphi\in C_c(\Omega;\R^m).
\]
We have that $|f\mu|= |f|\mu$ here. We will refer to $\mu$ as {\it finite} if $|\mu|(\Omega)<\infty$. We can also use the 1-1 correspondence to note that if, for all $\varphi\in C_c(\Omega)$, $\int_\Omega \varphi\cdot\dif\mu=0$, then the range of $\mu$ is $\{\bm{0}\}$. We refer to this as the {\it fundamental lemma of the calculus of variations}. In particular, this means that if $\mu$ is a Radon measure on $\Omega$, $f\in L^1_{\text{loc}}(\Omega;\R^m,\mu)$ and 
\[
\int_{\Omega}(\varphi\cdot f)\dif\mu = 0, \qquad\text{ for all } \varphi\in C_c(\Omega;\R^m),
\]
then the range of $\mu$ is $\{\bm{0}\}$ or $f=\bm{0}$, $\mu$-a.e.

If $\{\mu_k\}_{k\in\N}$ and $\mu$ are $\R^m$-valued Radon measures on $\Omega$, we say that $\mu_k$ weak* converges to $\mu$ if
\[
\int_\Omega\varphi\cdot\dif\mu = \lim\limits_{k\to\infty}\int_\Omega\varphi\cdot\dif\mu_k, \quad \text{ for all } \varphi\in C_c(\Omega)
\]
and we denote this by $\mu_k\wk\mu$. We collect some results about weak* convergence of vector-valued Radon measures in the following proposition.
\begin{proposition}\label{weak* results}
    Suppose $\{\mu_k\}_{k\in\N}$ and $\mu$ are $\R^m$-valued Radon measures on $\Omega$.
    \begin{enumerate}
        \item $\mu_k\wk\mu$ if and only if for any bounded Borel set $A\subset\Omega$ with $|\mu|(\partial A)=0$, we have that $\lim\limits_{k\to\infty}\mu_k(A)=\mu(A)$.
        \item If $\sup\limits_{k\in\N}|\mu_k|(K)<\infty$ for all compact $K\subset\Omega$, then there exists a $\R^m$-valued Radon measure on $\Omega$, $\nu$ and a subsequence of $\{\mu_k\}_{k\in\N}$ that weak* converges to $\nu$. Furthermore, if $\sup\limits_{k\in\N}|\mu_k|(\Omega)<\infty$, then $|\nu|(\Omega)<\infty$.
        \item Considering the $\varepsilon$-regularization of $\mu$, defined by
        \[
        \mu_\varepsilon(\varphi)=\int_\Omega\varphi(x)\cdot\int_\Omega\rho_\varepsilon(x-y)\dif\mu(y)\dif x, \qquad \text{ for } \varphi\in C_c(\Omega;\R^m),
        \]
        we have that as $\varepsilon\to 0^+$, $\mu_\varepsilon\wk\mu$ and $|\mu_\varepsilon|\wk|\mu|$.
    \end{enumerate}
\end{proposition}
Note that by the density of $C^\infty_c(\Omega)$ in $\Lip_c(\Omega)$, we may consider our test functions in the above results to be compactly supported Lipschitz functions rather than compactly supported continuous functions.

We remark upon the idea of {\it foliations by Borel sets}, that is: If $\{A_t\}_{t\in I}$ is a disjoint family of Borel sets in $\Omega$, indexed over some set $I$, and $\mu$ is a Radon measure on $\Omega$, then $\mu(A_t)>0$ for at most countably many $t\in I$.
\subsection{Sets of finite perimeter}\label{finite perimeter}
Let $E$ be a Lebesgue measurable set in $\Omega$. We say that  $E$ is a {\it set of locally finite perimeter in $\Omega$} if for every compact set $K\subset\Omega$, we have 
\[
\sup\left\{\int_E\Div T(x)\dif x\quad |\quad T\in C^1_c(\Omega,\R^n), \spt T\subset K, \sup_{\Omega} |T|\leq 1\right\} <\infty.
\]
If we remove the conditions involving $K$, we say that $E$ is a {\it set of finite perimeter in $\Omega$}.
We can see that $E$ is a set of locally finite perimeter in $\Omega$ if and only if there exists a $\R^n$-valued Radon measure $\nabla\chi_E$ on $\Omega$ such that
\[
\int_E \Div T = -\int_{\Omega}T\cdot \dif\,(\nabla\chi_E), \quad\text{ for all }\, T\in C^1_c(\Omega,\R^n).
\]
Moreover, $E$ is a set of finite perimeter if and only if $|\nabla\chi_E|(\Omega)<\infty$.

Given $x\in\Omega$, if the limit
		\[
		\theta_n(E)(x)=\lim\limits_{r\to 0^+}\frac{|E\cap B(x,r)|}{\omega_nr^n}
		\]
		exists, it is called the {\it $n$-dimensional density of $E$ at $x$}. For $t\in[0,1]$, we let
		\[
		E^{t}\coloneqq\{x\in\Omega \, |\, \theta_n(E)(x)=t\}.
		\]
$E^1$ and $E^0$ are called the {\it measure theoretic interior} and {\it exterior} of $E$ in $\Omega$, respectively, and we have that $|E\Delta E^1|=0$ and $|(\Omega\setminus E)\Delta E^0|=0$.
The {\it reduced boundary} of a set of locally finite perimeter in $\R^n$, $E$, is the set
		\[
		\partial^*E\coloneqq
		\left\{x\in\spt(\nabla\chi_E)\,\middle |\,\lim\limits_{r\to 0^+}\frac{\nabla\chi_E(B(x,r))}{|\nabla\chi_E|(B(x,r))}\quad\text{ exists and belongs to }S^{n-1}\right\}
		\]
        which is a subset of the topological boundary, $\partial E$.
		We define the Borel function \newline$\n_E:\partial^*E\to S^{n-1}$ by setting
		\[
		\n_E(x)=\lim\limits_{r\to 0^+}\frac{\nabla\chi_E(B(x,r))}{|\nabla\chi_E|(B(x,r))}, \quad \text{ for } x\in\partial^*E
		\]
		and call $\n_E$ the {\it (measure-theoretic) inner unit normal to $E$}. For all $x\in\partial^*E$, we let $T_xE\coloneqq\n_E(x)^\perp$ and refer to this as the {\it tangent space to $E$ at $x$}.
        
		By the Lebesgue-Besicovitch differentiation theorem
		\[
		\nabla\chi_E=\n_E|\nabla\chi_E|\res\partial^*E.
        \]
We know from De Giorgi's structure theorem that
\begin{equation}\label{De Giorgi formulae}
    \nabla\chi_E= \n_E\mathcal{H}^{n-1}\res\partial ^* E, \qquad |\nabla\chi_E|=\mathcal{H}^{n-1}\res\partial ^* E
\end{equation}
		and $\rb E$ is {\it locally $\h^{n-1}$-rectifiable}. Rather than define what it means for $\rb E$ to be locally $\h^{n-1}$-rectifiable, it will suffice for our purposes to note some consequences. In particular, we have that for any $x\in\rb E$,
    \begin{equation*}
        E_{x, r}\coloneqq\left(\frac{E-x}{r}\right) \overset{\text{loc}}{\longrightarrow }\{y\cdot\bm{\nu}_E(x)>0\},
    \end{equation*}
    as $r\to 0^+$. That is, for any compact $K\subset\Omega$,
    \begin{equation}\label{rectifiability}
        \lim\limits_{r\to 0^+}\big|[E_{x,r}\Delta \{y\cdot\bm{\nu}_E(x)>0\}]\cap K\big|=0.
    \end{equation}
    Further consequences are that for any $x\in\rb E$,
    \begin{equation}\label{boundary has half density}
        \lim\limits_{r\to 0^+}\frac{|E\cap B(x,r)|}{\omega_nr^n}=\frac{1}{2}
    \end{equation}
    and
    \begin{equation}\label{rectifiability in terms of measure}
        \lim\limits_{r\to 0^+}\frac{(\h^{n-1}\res\rb E)(B(x,r))}{\omega_{n-1}r^{n-1}}=1.
    \end{equation}
In particular, \eqref{boundary has half density} implies that
\begin{equation}\label{boundary has half density 2}
    \rb E\subset E^{1/2}.
\end{equation}
Federer's theorem tells us that
\begin{equation}\label{Federer}
    \h^{n-1}(\Omega\setminus(E^0\cup\rb E\cup E^1))=0.
\end{equation}

We let $\rho$ denote a {\it standard symmetric mollifying kernel}, that is, a function in $C_c^\infty(\R^n)$ that is supported in the unit ball, satisfies $\int_{\R^n}\rho(x)\dif x=1$, has a range in $[0,1]$ and is radially symmetric. For $y\in\R^n$ and $\varepsilon>0$, we let $\rho_\varepsilon(y)\coloneqq\frac{1}{\varepsilon^n}\rho(\frac{y}{\varepsilon})$. Mollifications may not be well-defined on all of $\Omega$, so for $\delta>0$, we let
\[
\Omega^\delta\coloneqq\{x\in\Omega : \text{dist}(x,\partial\Omega)>\delta\}
\]
and for $f\in L^1_{\text{loc}}(\Omega)$, we let $f_\delta\coloneqq f*\rho_\delta$ be the mollification of $f$ with $\rho_\delta$, which is defined on $\Omega^\delta$. Since $f_\delta\in C_c(\Omega^\delta)$, we may extend $f_\delta$ by zero on $\Omega$ and it will lie in $C_c(\Omega)$.

We let $BV(\Omega)$ denote the functions of {\it bounded variation} on $\Omega$, that is, the functions $u\in L^1(\Omega)$ whose distributional derivatives are $\R^n$-valued Radon measures with finite total variation. We denote this measure by $\nabla u$ and $|\nabla u|$ is its total variation. So, in particular the characteristic function of a set of finite perimeter $E\subset\Omega$ lies in $BV(\Omega)$.

A function $u\in BV(\Omega)$ has precise representative $u^*$ where $u=u^*$ $\lb^n$-a.e. and the mollification sequence of $u$, $u_k$ converges to $u^*$ $\h^{n-1}$-a.e. Note that where this convergence occurs is dependent on the mollifying kernel that we use, so we just use one standard symmetric mollifying kernel $\rho$ in order to be consistent. It can be shown that for $\h^{n-1}$-a.e. $x\in\Omega$,
\begin{equation}\label{precise rep formula}
    u^*(x)=\lim\limits_{r\to 0^+}\frac{1}{|B(x,r)|}\int_{B(x,r)}u(y)\dif y.
\end{equation}
If $u=\chi_E$ for some set of finite perimeter $E$, we can say more: There is a $\h^{n-1}$-null set, $\mathcal{N}$, such that for all $y\notin\mathcal{N}$,
    \[
    \chi_E^*(y)=
    \begin{cases}
        1,\quad &y\in E^1,\\
        \frac{1}{2},\quad &y\in \rb E,\\
        0,\quad &y\in E^0.
    \end{cases}
    \]
This is a consequence of \eqref{boundary has half density 2}, \eqref{Federer} and \eqref{precise rep formula}.
    
    Later we will consider $g\in BV(\Omega)$ and consider a sequence $\{g_k\}_{k\in\N}\subseteq C^\infty_c(\Omega)$ as constructed in Theorem 5.3.3 of \cite{ZiemerWeak} that not only satisfies $g_k\to g$ in $L^1\left(\Omega\right)$, as $k\to\infty$, but $|\nabla g_k|(\Omega)\to|\nabla g|(\Omega)$ as well. This is not generally satisfied by mollifications of $g$ and indeed looking at the construction in \cite{ZiemerWeak}, we see that $g_k = \sum\limits_{i\in\N}((f_ig)*\rho_{\varepsilon_k})$, where the $f_i$ are functions arising from a partition of unity and $\varepsilon_k\to 0$. However, for any $x\in\Omega$, if we look at a small enough neighborhood of $x$ and let $k$ become large enough, we will see that $g_k$ is really (up to re-labelling of the $f_i$) just the finite sum $\sum_{i=1}^N((f_i g)*\rho_{\varepsilon_k})= (\sum_{i=1}^Nf_i g)*\rho_{\varepsilon_k} = g*\rho_{\varepsilon_k}$. That is to say, outside of the same $\h^2$-null set, we still have convergence of $g_k$ to $g^*$.
    
\subsection{Weak* topology}\label{functional analysis}
We recall (from section 8.4 of \cite{TorresZiemerBook}, for instance) that if $X$ is a normed linear space, then we may consider the {\it weak* topology} on $X$. That is, the smallest topology on the dual of $X$, $X^*$ for which each element in the range of the natural embedding, $\Phi:X\to (X^*)^*$, is continuous.
From Banach-Alaoglu's theorem, it follows that, for any $\sigma$-finite measure space $(M,\mu)$, we may consider the normed linear space $X\coloneqq L^1(M,\mu)$ and if we have a bounded sequence in $X^*$, there exists a subsequence that converges in the weak* topology. This amounts to the following: If $\{f_k\}_{k\in\N}$ is a bounded sequence in $L^\infty(M,\mu)$, then there exists a subsequence $\{f_{k_j}\}_{j\in\N}$ and a $f\in L^\infty(M,\mu)$ such that for all $g\in L^1(M,\mu)$, 
\[
\lim_{k\to\infty}\int_Mf_{k_j}g\dif\mu = \int_Mfg\dif\mu.
\]
Here, we say that as $k\to\infty$,
\[
f_{k_j}\rightharpoonup f \qquad \text{ in } L^\infty(M,\mu).
\]
\section{Curl-measure fields and tangential traces}\label{cm fields section}
To motivate our study of curl-measure fields and tangential traces, we consider the following example.
\begin{example}\label{example0}
    Let $P\coloneqq\{(x,y,z)\in\R^3 : x+y+z=0\}$ and let 
    $\F:\R^3\setminus P\to\R^3$ be defined by
    \[
    \F(x,y,z)= \sin\left(\frac{1}{x+y+z} \right)(1,1,1).
    \]
    We note that, on $\R^3\setminus P$, $\Curl\F = \bm{0}$. Let $C$ be the open interior of a unit cube in $\R^3$ with one face, $S$, contained in $P$. We specify that $C\subset \{(x,y,z)\in\R^3 : x+y+z>0\}$. Clearly, $\int_C\Curl\F\dif\lb^3 = \bm{0}$, and hence we can compute the left side of \eqref{curl div theorem}. 
    However, when attempting to generalize the right hand side of \eqref{curl div theorem}, it is not clear what the integrand over $\partial C$ should be, since $\F\times\n_{\partial C}$ is not defined on $S$. The vector field $\F$ cannot be continuously extended to $S\subset\partial C$ and its component functions are not of bounded variation.
    
    However, if we instead consider $\Curl\F$ and $\F\times\n_{\partial C}$ in the distributional sense described in this section, we will be able to obtain a more general Gauss-Green formula. The {\it tangential trace}, which plays the role of $\F\times\n_{\partial C}$ from the smooth case, will be a bounded function that is {\it tangential} to $\partial^* C$ at $\h^2$-almost every point of $\partial^* C$.
    At the end of the paper, we will revisit this example and demonstrate these facts. 
\end{example}

\subsection{Curl-measure fields}
\begin{definition}
    For an open set $\Omega\subset\R^3$, a vector field $\F\in L^p\left(\Omega;\R^3\right)$, $1\leq p\leq\infty$, is called a {\it curl-measure field} if $\mu\coloneqq \Curl \F$ is a Radon measure with finite total variation in the sense of distributions. We write this as $\F\in\mathcal{CM}^p(\Omega)$. Thus, for $\varphi\in C^\infty_c(\Omega)$, we have
    \[
    \mu(\varphi)\coloneqq\left(\Curl\F\right)(\varphi)= -\int\limits_\Omega\F\times \nabla\varphi\,\dif y.
    \]
    Recall that the total variation of $\mu$ is a nonnegative measure which, for any open set $W\subseteq\Omega$, is defined as
    \begin{align*}
        |\mu|(W)\coloneqq &\sup\left\{\mu(\varphi): |\varphi|\leq 1, \varphi\in C^\infty_c(W)\right\}\\
        =&\sup\left\{\int\limits_\Omega\F\times \nabla\varphi\,\dif y: |\varphi|\leq 1, \varphi\in C^\infty_c(W)\right\}
    \end{align*}
    and for any arbitrary set $V\subseteq \Omega$, 
    \[
    |\mu|(V)\coloneqq\inf\{|\mu|(A) : V\subseteq A \text{ and } A \text{ is open}\}.
    \]
    Thus, $|\mu|(\Omega)<\infty$, here.
    A vector field $\F\in\mathcal{CM}^p_{\text{loc}}(\Omega)$ means that, for any $W\Subset\Omega$, $\F\in\mathcal{CM}^p(W)$.
\end{definition}
As seen in the examples in Section 3 of \cite{CurlMeasure25}, if $p\in[1,\infty)$, it is not always possible to obtain a generalization of \eqref{mainformula2} where the object representing $\F\times\n_{\partial V}$ is a locally integrable function. This object is the {\it tangential trace}, which we will define shortly. However, this limitation does not apply when $p=\infty$. Consequently, this paper will focus on the case where $\F\in\mathcal{CM}^\infty(\Omega)$ and $E$ is some set of finite perimeter. An explicit example of this is described in Example \ref{example2}. However, we first give the following simpler example where $\F\in\mathcal{CM}^\infty_{\text{loc}}(\Omega)$ and $E$ is a set of locally finite perimeter.
\begin{example}\label{example1}
    Let $\Omega=\R^3$, let $E$ be the half-space $\left\{(x,y,z) : z>0\right\}$ and for some constant vectors $\F_1, \F_2\in\R^3$ let $\F$ be defined by
    \[
    \F(x,y,z)=
    \begin{cases}
        \F_1, \quad z>0,\\
        \F_2, \quad z\leq 0.
    \end{cases}
    \]
    In this case, $\partial E = \left\{(x,y,z) : z=0\right\}$, $\n_{\partial E}=(0,0,1)$ and for any $\varphi\in C_c^\infty(\Omega)$, we have that
    \begin{align*}
        \left(\Curl\F\right)(\varphi) &= \int\limits_\Omega\F\times \nabla\varphi\,\dif y
        = \int_E\F_1\times\nabla\varphi\dif y +\int_{\Omega\setminus E}\F_2\times\nabla\varphi\dif y\\
        &= \int_{\partial E}\F_1\times\varphi(-\n_{\partial E})\dif\h^2 +\int_{\partial E}\F_2\times\varphi\n_{\partial E}\dif\h^2\\
        &=\int_{\partial E}\varphi (\F_2-\F_1)\times\n_{\partial E},
    \end{align*}
    where the third equality follows from the classical Gauss-Green formula \eqref{mainformula2}, that $\varphi$ is compactly supported and that, outside of $\partial E$, $\Curl\F=\bm{0}$. Thus,
    \[
    \Curl\F = ((\F_2-\F_1)\times\n_{\partial E})\h^2\res \partial E
    \]
    and since $|((\F_2-\F_1)\times\n_{\partial E})\h^2\res \partial E|=|\F_2-\F_1|\h^2\res \partial E$, we see that $\F\in\mathcal{CM}^\infty_{\text{loc}}(\Omega)$.
\end{example}
We next prove a result concerning $\mathcal{CM}^\infty(\Omega)$ that will be used in both Sections \ref{product method section} and \ref{functional method section}. 
\begin{proposition}\label{absolute continuity}
    If $\F\in\mathcal{CM}^\infty\left(\Omega\right)$, then $\Curl\bm{F}\ll\h^2$ in $\Omega$. That is, if $A\subset\Omega$ is a Borel measurable set such that $\h^2(A)=0$, then $|\Curl\F|(A)=0$.
\end{proposition}
\begin{proof}
    Let $\mu=\Curl\bm{F}$. From Theorem 4.2 and Corollary 4.3 in \cite{PhucTorres17}, our conclusion holds if there exists $C>0$ such that, for any smooth, open and bounded set $U\Subset\Omega$, 
    \[
    |\mu(U)|\leq C\,\h^2(\partial U).
    \]
    But Theorem 3.9 in \cite{CurlMeasure25} (the trace theorem for bounded Lipschitz domains) tells us that there exists $\bm{G}\in L^\infty(\partial U,\R^3)$ such that for every $\varphi\in C^1_c(\Omega)\supseteq C^\infty_c(\Omega)$,
    \[
    \int\limits_{\partial U}\varphi\bm{G}\,\dif\h^2 = \int\limits_U\varphi \,\dif\mu - \int\limits_U\bm{F}\times\nabla\varphi\,\dif x.
    \]
    Choosing $\varphi\in C^1_c(\Omega)$ such that $\varphi\equiv 1$ in a neighborhood of $\overline{U}$, we see that the previous equation becomes
    \[
    \int\limits_{\partial U}\bm{G}\,\dif\h^2 = \int\limits_U\dif\mu,
    \]
    so
    \[
    |\mu(U)|\leq ||\bm{G}||_\infty\h^2(\partial U)
    \]
    and we are done.
\end{proof}

\subsection{Tangential traces}\label{tangential traces section}
\begin{definition}
\label{aqui}
Let $\F\in\mathcal{CM}^\infty(\Omega)$, and let $A\Subset\Omega$ be a Borel set. The {\it  tangential trace} of $\F$ on the boundary of $A$ is defined as
\begin{equation}
\label{ladef}
\langle \F\times\n,\varphi\rangle_{\partial A} := \int_{A}\varphi \dif\, (\Curl\F) - \int_{A}\F\times\nabla\varphi \, \dif\,x,\quad \text{ for } \varphi\in\Lip_c(\Omega)
\end{equation}
\end{definition}
Note that if $A \Subset \Omega$ is a bounded open set with smooth boundary and $\F$ is a smooth vector field on $\Omega$, then the left side of \eqref{ladef} is represented as  $\int_{\partial A}\varphi(y)\F(y)\times\n_A(y)\dif\h^2(y)$. However, for  $\F\in \mathcal{CM}^\infty(\Omega)\subset L^\infty(\Omega)$, the values of $\F$
on the boundary of an arbitrary  set are not defined in general, and therefore a notion of tangential trace needs to be established. 

We remark that, if the Borel sets $G$ and $H$ differ by a set of Lebesgue measure zero, we can still have $\langle \F\times\n,\cdot \rangle_{\partial G} \neq \langle \F\times\n,\cdot \rangle_{\partial H}$ as distributions. This is because the first integral in \eqref{ladef} may take different values when the measure $\Curl \F$ concentrates on lower dimensional sets. 

In this paper, we consider $E \Subset \Omega$ a set of finite perimeter and the tangential traces for the two Borel sets $G= E^1$ and $H=E^1 \cup \partial^* E$.
We will refer to $\langle\F\times\n,\cdot\rangle_{\partial (E^1 \cup \partial^* E)}$ and $\langle\F\times\n,\cdot\rangle_{\partial (E^0)}$ as the {\it interior and exterior tangential traces on the boundary of $E$}, respectively.
We will show that these tangential traces can be {\it represented as essentially bounded functions on $\rb E$}. By this, we mean that there exists $\mF_i\times\n_E,\,\mF_e\times\n_E\in L^\infty(\rb E)$ (see Theorems \ref{first trace theorem} and \ref{Second trace theorem}) such that, for all $\varphi\in\Lip_c(\Omega)$,
\begin{equation}\label{int}
\langle\F\times\n,\varphi\rangle_{\partial (E^1)} = \int_{\rb E}\varphi (x)\, (\mF_i\times\n_E) (x) \dif\h^2 (x) \quad 
\end{equation}
and
\begin{equation}\label{ext}
\langle\F\times\n,\varphi\rangle_{\partial (E^1 \cup \partial^* E)} = \int_{\rb E}\varphi (x)\, (\mF_e\times\n_E)(x) \dif\h^2 (x).
\end{equation}
Due to \eqref{int} and \eqref{ext}, we will often refer to $\mF_i\times\n_E$ and $\mF_e\times\n_E$ themselves as the interior and exterior tangential traces on the boundary of $E$, respectively. The theorems that tell us how we may represent the traces are referred to as {\it trace theorems} and the formulae that describe these representations (such as \eqref{int} and \eqref{ext}) are called {\it Gauss-Green formulae}. The reason for use of the word {\it exterior} here is as follows. Note that
\begin{align*}
    \langle \F\times\n,\varphi\rangle_{\partial (E^1 \cup \partial^*E)} &=  \int_\Omega\varphi\dif\,(\Curl\F)-\int_{E^0}\varphi\dif\,(\Curl\F)-\int_\Omega\F\times\nabla\varphi+\int_{E^0}\F\times\nabla\varphi\\
    &=\left(\int_\Omega\varphi\dif\,(\Curl\F)-\int_\Omega\F\times\nabla\varphi\right)-\left(\int_{E^0}\varphi\dif\,(\Curl\F)-\int_{E^0}\F\times\nabla\varphi\right)\\
    &= \bm{0} - \langle \F\times\n,\varphi\rangle_{\partial (E^0)},
\end{align*}
where the first equality follows by considering \eqref{ext} and the partition of $\Omega$ into $E^1$, $E^0$, $\rb E$ and a $\h^2$-null set. The third equality may be shown using results seen later in this paper. Specifically, Theorem \ref{ProductRule} tells us that $\varphi\F\in \mathcal{CM}^\infty(\Omega)$ and Proposition \ref{curl vanishes} gives us that $\Curl(\varphi\F)(\Omega)=\bm{0}$. Then, using the definition of $\mathcal{CM}^\infty(\Omega)$ to see that $\Curl(\varphi\F)=\varphi\Curl\F -\F\times\nabla\varphi\lb^3$, we conclude that $\int_\Omega\varphi\dif\,(\Curl\F)-\int_\Omega\F\times\nabla\varphi = \Curl(\varphi\F)(\Omega)=\bm{0}$.
Thus,
\begin{equation}
\label{ladef2}
\langle \F\times\n,\varphi\rangle_{\partial (E^1 \cup \partial^*E)} =-\left(\int_{E^0}\varphi \dif\, (\Curl\F) - \int_{E^0}\F\times\nabla\varphi\right) = - \langle \F\times\n,\varphi\rangle_{\partial (E^0)}
\end{equation}
and we could have instead defined the exterior tangential trace on the boundary of $E$ as $-\langle \F\times\n,\cdot\rangle_{\partial (E^0)}$. 

The notation here is chosen because it is reminiscent of the {\it smooth case} where $\F$ is a smooth vector field on $\Omega$ and $E\Subset\Omega$ is open with a smooth boundary. In this case, we would have that for $\varphi\in\Lip_c(\Omega)$,
\begin{align*}
    \int_{\partial E}\varphi\F\times\n_E \dif\h^2 &= \int_{E}\varphi \Curl\F\dif\,x - \int_{E}\F\times\nabla\varphi \, \dif\,x\quad \text{ and } \\
-\int_{\partial E}\varphi\F\times\n_E \dif\h^2 &= \int_{\Omega\setminus E}\varphi \Curl\F\dif\,x - \int_{\Omega\setminus E}\F\times\nabla\varphi \, \dif\,x.
\end{align*}
Thus, in this case, we see that the interior and exterior tangential traces are $\F\times\n_E$ and $-\F\times\n_E$, respectively.
For applications, such as those discussed in the introduction, we do not restrict ourselves to this smooth case since we wish to allow the existence of a nonzero {\it jump of the tangential traces}, $\mF_i\times\n_E-\mF_e\times\n_E$. For example, in the analysis of vortex sheets seen in \cite{CurlMeasure25}, $\F$ represents the velocity of a fluid where there is a difference in the tangential velocity of the fluid along $\rb E$, such as one may observe in the slippage of one layer of fluid over another. Here, $\rb E$ is called the {\it vortex sheet}. $\mF_i\times\n_E$ and $\mF_e\times\n_E$ represent $\F\times\n_E$ along $\rb E$ coming from inside and outside $E$, respectively. Note that $\F\times\n_E$ is {\it not} the projection of $\F$ onto the tangent space of $E$. That value, also known as the {\it tangential velocity}, may be readily computed as $\n_E\times(\n_E\times\F)$ since
\[
\F-(\n_E\cdot\F)\n_E = (\n_E\cdot\n_E)\F-(\n_E\cdot\F)\n_E = \n_E\times(\F\times\n_E).
\]
Thus, this jump {\it determines} the difference in tangential velocity across $\rb E$, up to a rotation of $90^\circ$ in the tangent plane.
Subtracting \eqref{int} from \eqref{ext} and noting that $\lb^3(\rb E)=0$ gives us that
\begin{equation}\label{jump}
    \left(\mF_e\times\n_E-\mF_i\times\n_E\right)\h^2\res\rb E = \Curl\F\res\rb E.
\end{equation}
Hence, the jump in tangential velocity across $\rb E$ is determined by $\Curl\F\res\rb E$. In fact, it is the Radon-Nikodym derivative which we knew to exist from Proposition \ref{absolute continuity}. Let us look at this in the context of an explicit example.
\begin{example}\label{example2}
     Let $\Omega=\R^3$, $E=B(0,1)$ and $\F\coloneqq \chi_E\F_1+\chi_{\Omega\setminus E}\F_2$, where
     \[
     \F_1(x,y,z) =(-y,x,0) \quad \text{ and } \quad \F_2(x,y,z) =(-2y,2x,0), \quad \text{ for }(x,y,z)\in\Omega.
     \]
     In this case, $\rb E = \partial E = \left\{(x,y,z) : x^2+y^2+z^2=1\right\}$, $\n_{\partial E}=(-x,-y,-z)$ and for any $\varphi\in C_c^\infty(\Omega)$, we have that
    \begin{align*}
        \left(\Curl\F\right)(\varphi) &= \int\limits_\Omega\F\times \nabla\varphi\,\dif y
        = \int_E \F_1\times\nabla\varphi\dif y + \int_{\Omega\setminus E}\F_1\times\nabla\varphi \dif y\\
        &
        \begin{aligned}
            =-\int_{\partial E}\varphi\F_1\times\n_{\partial E}& \dif\h^2+\int_E\varphi\Curl\F_1\dif y \\
            &+ \int_{\partial E}\varphi\F_2\times\n_{\partial E} + \int_{\Omega\setminus E}\varphi\Curl\F_2\dif y
        \end{aligned} \\
        &=\int_{\partial E}\varphi(\F_2-\F_1)\times\n_{\partial E} \dif\h^2+\int_\Omega\left(\varphi\chi_E\Curl\F_1+\chi_{\Omega\setminus E}\Curl\F_2\right)\dif y.
    \end{align*}
    where the second equality follows from classical Gauss-Green formula \eqref{mainformula2} and that $\varphi$ is compactly supported. Thus,
    \begin{equation}\label{examplecurl}
        \Curl\F = \left(-xz,-yz,1-z^2\right)\h^2\res\rb E +\left((0,0,2)\chi_E+(0,0,4)\chi_{\Omega\setminus E}\right)\lb^3\res\Omega.
    \end{equation}
    We consider what the terms in this equation represent in the context of $\F$ being the velocity of a fluid in $\Omega$. As defined here, $\F$ is the velocity of a fluid that rotates about the $z$-axis. Inside $E$, it has constant angular velocity and outside of $E$, the angular velocity has been doubled. The terms $(0,0,2)\chi_E\lb^3\res\Omega$ and $(0,0,4)\chi_E\lb^3\res\Omega$ account for the curl of $\F$ being constant in $E$ and $\Omega\setminus E$, respectively. For any point on $x\in\rb E$, if we approach $x$ from either side, the tangential velocity of $\F$ will be in the same direction, but the speed will be twice as fast when approaching from $\Omega\setminus E$. This is why we have the term $\left(-xz,-yz,1-z^2\right)\h^2\res\rb E$ in \eqref{examplecurl}, which accounts for the jump in tangential velocity. In the notation used above, we have that 
    \[
    \mF_e\times\n_E-\mF_i\times\n_E = (-xz,-yz,1-z^2).
    \]
    Fixing $z=z_0$ and looking at points $(x,y,z_0)\in\rb E$, that is, looking at latitudinal lines of $\rb E$, we see that this term is consistent with our understanding. Noting that
    \[
    \sqrt{\left(-xz_0\right)^2+\left(-yz_0\right)^2+\left(1-(z_0)^2\right)^2}=\sqrt{1-(z_0)^2}
    \]
    allows us to see that the magnitude of the jump in tangential velocity is constant for a fixed $z_0$. We expect this since there is a constant jump in angular velocity with respect to the $z$-axis. 
\end{example}
\section{Product Rule Approach}\label{product method section}
Here, we use the approach used in \cite{ChenTorres05} to obtain a trace theorem. We use some ideas seen in \cite{ExtDiv25} and \cite{ComiThesis} to more readily acquire the results that we will use in the final section.
For the rest of the paper, we assume that $n=3$, so $\Omega\subseteq\R^3$.
\begin{theorem}\label{ProductRule}
    Let $\F\in\mathcal{CM}^\infty\left(\Omega\right)$ and let $g\in BV\left(\Omega\right)\cap L^\infty(\Omega)$. Then
    \[
    \Curl(g\F) = g^*\Curl\F-\overline{\F\times\nabla g},
    \]
    where $g^*$ is the precise representative of $g$ and $\overline{\F\times\nabla g}$ denotes a $\R^3$-valued Radon measure defined by
    \[
    \int_\Omega\phi \dif\, (\overline{\F\times\nabla g}) = \lim\limits_{k\to\infty}\int_\Omega\phi\F\times\nabla g_k, \qquad\text{ for all } \phi\in C_c(\Omega),
    \]
    where $\left\{g_k\right\}_{k\in\N}$ is a particular sequence of $C^\infty_c(\Omega)$ mollifications of $g$ with the property that $g_k\to g$ in $L^1\left(\Omega\right)$ and $|\nabla g_k|(\Omega)\to|\nabla g|(\Omega)$ (existence of this sequence and why it converges $\h^2$-a.e. to $g^*$ is discussed in Section \ref{finite perimeter}).
    Moreover, $\overline{\F\times\nabla g}\ll |\nabla g|$.
\end{theorem}
\begin{proof}
    So that mollifications are well-defined, we will initially consider $\Omega^\delta$ for some $\delta>0$. For $0<\varepsilon<\delta$, let $\bm{F}_\varepsilon$ be the mollification of $\bm{F}$ and let $\mu\coloneqq\Curl\bm{F}$. Since the $\bm{F}_\varepsilon$ are smooth, the classical product rule gives us that on $\Omega^\delta$,
    \begin{equation*}
        \Curl(g_k\bm{F}_\varepsilon) = g_k\Curl(\bm{F}_\varepsilon)-\bm{F}_\varepsilon\times\nabla g_k
    \end{equation*}
    or 
    \begin{equation}\label{classicprod}
        \bm{F}_\varepsilon\times\nabla g_k = -\Curl(g_k\bm{F}_\varepsilon) + g_k\Curl(\bm{F}_\varepsilon).
    \end{equation}
    Let $\varphi\in C_c(\Omega^\delta)$. Considering $\F\times\nabla g_k$ as a $\R^3$-valued Radon measure on $\Omega^\delta$, there exists a $D>0$ such that, as $\varepsilon\to 0^+$,
    \begin{align*}
        |(\bm{F}_\varepsilon\times\nabla g_k)(\varphi)-(\bm{F}\times\nabla g_k)(\varphi)| = \left|\,\int\limits_{\Omega^\delta}\varphi (\bm{F}_\varepsilon-\bm{F})\times\nabla g_k\right|
        \leq D\int\limits_{\Omega^\delta}|\bm{F}_\varepsilon-\bm{F}|
        &\to 0,
    \end{align*}
    and the existence of $D$ is due to the compact support of $\varphi$. Hence, as measures on $\Omega^\delta$,
    \begin{equation}\label{first}
        \bm{F}_\varepsilon\times\nabla g_k\wk\bm{F}\times\nabla g_k.
    \end{equation} 
    Next, note that
    \begin{align*}
        (\Curl\bm{F})_\varepsilon(\varphi)= \mu_\varepsilon(\varphi) &= \int\limits_{\Omega^\delta}\varphi(x) (\mu*\rho_\varepsilon)(x) \dif x
        = \int\limits_{\Omega^\delta}\varphi(x) \int\limits_{\Omega^\delta}\rho_\varepsilon(x-y)\dif\mu(y) \dif x\\
        &=\int\limits_{\Omega^\delta}\int\limits_{\Omega^\delta}\varphi(x) \rho_\varepsilon(x-y)\dif x\dif\mu(y)
        =\int\limits_{\Omega^\delta}\varphi_\varepsilon(y)\dif\mu(y)
        = \mu(\varphi_\varepsilon)\\
        &= -\int\limits_{\Omega^\delta}\left(\F\times \nabla\varphi_\varepsilon\right)(y)\,\dif y
        = -\int\limits_{\Omega^\delta}\left(\F\times\left( \rho_\varepsilon*\nabla\varphi\right)\right)(y)\,\dif y\\
        &= -\int\limits_{\Omega^\delta}\left(\F_\varepsilon\times \nabla\varphi\right)(y)\,\dif y
        = (\Curl (\F_\varepsilon))(\varphi),
    \end{align*}
    where the second-to-last equality follows from Fubini's theorem, the bilinearity of the cross product and the linearity of integration. This means that $\Curl (\bm{F}_\varepsilon) =(\Curl \bm{F})_\varepsilon$. We have from Proposition \ref{weak* results}.3 that $(\Curl \bm{F})_\varepsilon =\mu_\varepsilon\wk\mu$ as $\varepsilon\to 0^+$. It then follows that, as $\varepsilon\to 0^+$, 
    \begin{align*}
        (g_k\Curl(\bm{F}_\varepsilon))(\varphi) = \int\limits_{\Omega^\delta}g_k\varphi \dif\mu_\varepsilon
         \to \int\limits_{\Omega^\delta}g_k\varphi \dif\mu
        = (g_k\Curl\bm{F})(\varphi)
    \end{align*}
    That is,
    \begin{equation}\label{third}
        g_k\Curl(\bm{F}_\varepsilon)\wk g_k\Curl\bm{F}.
    \end{equation}

    Now consider $\varphi\in C^\infty_c(\Omega^\delta)$.
    Letting $S\coloneqq\spt\varphi$, we have that as $\varepsilon\to 0^+$,
    \begin{align*}
        |\Curl(g_k\bm{F}_\varepsilon)(\varphi)-\Curl(g_k\bm{F})(\varphi)| &= \left|\, \int\limits_{\Omega^\delta}g_k(\bm{F}_\varepsilon-\bm{F})\times\nabla\varphi\right|\\
        & \leq ||(g_k|\nabla\varphi|)||_{L^\infty(S)} \int\limits_{\Omega^\delta}|\bm{F}_\varepsilon-\bm{F}|
         \to 0,
    \end{align*}
    So,
    \begin{equation}\label{second}
        \Curl(g_k\F_\varepsilon)\overset{\mathcal{D'}}{\longrightarrow}\Curl(g_k\F),
    \end{equation}  
     that is, we have convergence in the sense of distributions on $\Omega^\delta$.
    Applying \eqref{first},\eqref{second} and \eqref{third} to \eqref{classicprod} gives
    \begin{equation}\label{no epsilon}
        \bm{F}\times\nabla g_k = -\Curl(g_k\bm{F}) + g_k\Curl\bm{F},
    \end{equation}
    which we can rearrange to see that $\Curl(g_k\F)$ is a measure on $\Omega^\delta$.

    Let $\nu_k$ be the measure $\bm{F}\times\nabla g_k$. Using that $|\bm{F}\times\nabla g_k\lb^3|=|\bm{F}\times\nabla g_k|\lb^3$, which follows from our discussion in Section \ref{measures}, we have that
    \begin{align*}
        |\nu_k|(\Omega^\delta)=\int\limits_{\Omega^\delta} \dif|\nu_k|
        = \int\limits_{\Omega^\delta} |\bm{F}\times\nabla g_k|\dif\lb^3 
        &\leq ||\F||_{L^\infty(\Omega^\delta)}|\nabla g_k|(\Omega^\delta)\\ 
        &\leq ||\F||_{L^\infty(\Omega^\delta)}\left(|\nabla g|(\Omega)+1\right)
    \end{align*}
    and this finite upper bound is independent of $k$. By Proposition \ref{weak* results}.2, there exists a finite Radon measure $\nu$ and subsequence $\{\nu_{h(k)}\}_{k\in\N}$ such that $\nu_{h(k)}\wk \nu$ as $k\to\infty$. Relabeling the subsequence as the original sequence, we have
    \begin{equation}\label{first2}
        \bm{F}\times\nabla g_k\wk\nu.
    \end{equation}
    Next, consider $\varphi\in C_c(\Omega^\delta)$. The dominated convergence theorem tells us that, as $k\to\infty$,
    \begin{align*}
        (g_k\Curl\bm{F})(\varphi)=\int\limits_{\Omega^\delta}\varphi g_k\dif\mu\to \int\limits_{\Omega^\delta}\varphi g^*\dif\mu.
    \end{align*}
    Note that we may apply the dominated convergence theorem since, as discussed in Section \ref{finite perimeter}, $g_k$ converges to $g^*$, $\h^2$-a.e. Since $\mu\ll\h^2$, we have that $g_k\to g^*$ $\mu$-a.e. The supports of the $g_k$ may vary with $k$, but the presence of $\varphi$ means that we are integrating over a fixed compact set. Additionally, for any $x\in\Omega^\delta$, $|g_k(x)|\leq||g||_{L^\infty(\Omega)}$, so we may use $\varphi g$ as our dominating function.
    So,
    \begin{equation}\label{third2}
        (g_k\Curl\bm{F})\wk g^*\Curl\bm{F}.
    \end{equation}
    Now, let $\varphi\in C^\infty_c(\Omega^\delta)$.
    \begin{align*}
        |\Curl(g_k\bm{F})(\varphi)-\Curl(g\bm{F})(\varphi)| &\leq  \int\limits_{\Omega^\delta} |(g_k-g)||\bm{F}\times\nabla\varphi| \to 0, \text{ as } k\to\infty,
    \end{align*}
    where the convergence follows from the fact that $\varphi$ has compact support, that $|\F|\in L^\infty(\Omega)$ and that $g_k\to g$ in $L^1(\Omega^\delta)$.
    Hence, we have
    \begin{equation}\label{second2}
        \Curl(g_k\bm{F})\overset{\mathcal{D}'}{\to}\Curl(g\bm{F}),
    \end{equation}
    which is convergence in the sense of distributions on $\Omega^\delta$.
    We now consider the map
    \begin{align*}
        \overline{\F\times\nabla g}:C_c(\Omega)&\to\R\\
        \phi&\mapsto\lim\limits_{k\to\infty}\int_{\Omega}\phi\F\times\nabla g_k
    \end{align*}
    which is well-defined since there exists a $\delta_0>0$ such that $\phi\in C_c(\Omega^{\delta_0})$. Since $\F\in L^\infty(\Omega)$ and $|\nabla g_k|(\Omega)\leq\left(|\nabla g|(\Omega)+1\right)$, this is a bounded linear functional on $C_c(\Omega)$ and hence defines a $\R^3$-valued Radon measure. For any $\phi\in C^\infty_c(\Omega)$, we have that  $\phi\in C_c(\Omega^{\delta_0})$ for some $\delta_0>0$ and
    \begin{align*}
        \lim\limits_{k\to\infty}\int_{\Omega}\phi\F\times\nabla g_k
        &=\lim\limits_{k\to\infty}\left[-(\Curl(g_k\F))(\phi)+(g_k\Curl\F)(\phi)\right]\\ 
        &= -(\Curl(g\F))(\phi)+(g^*\Curl\F)(\phi) .
    \end{align*}
    The first equality follows by considering our integrals as being over $\Omega^{\delta_0}$ instead of $\Omega$ and applying \eqref{no epsilon}. In the same way, the second equality follows from \eqref{third2} and \eqref{second2},  even though we are considering distributions on $\Omega$ rather than measures on $\Omega^{\delta_0}$. Thus, we have that
    \[
    \overline{\F\times\nabla g} = -\Curl(g\F)+g^*\Curl\F
    \]
    and by rearrangement, $\Curl(g\F)$ is a $\R^3$-valued Radon measure. Note that $\overline{\F\times\nabla g}$ is not necessarily the weak* limit of measures defined on all of $\Omega$.
    
    To prove absolute continuity, we proceed similarly to Theorem 3.2 in \cite{ChenFrid99}.
    Consider some set $A\subset\Omega$ such that $|\nabla g|(A)=0$.
    Since $\overline{\F\times\nabla g}$ is a Radon measure, it suffices to assume that $A$ is compact.
    For any $\varepsilon>0$, we claim that we may cover $A$ with $N$ balls, $B(x_1,r_1),...,B(x_N,r_N)$, such that
    \[
    A\subseteq\cup_{i=1}^N B(x_i,r_i),\quad r_i<\varepsilon,  \quad |\nabla g|(\partial B(x_i,r_i))=0 \quad \text{ and } \quad |\nabla g|(\cup_{i=1}^N B(x_i,r_i))<\varepsilon.
    \]
    To this end, we first find an open set $U\supset A$ such that $|\nabla g|(U)<\varepsilon$, using outer regularity. Since $A$ is compact, we may assume that $U$ is bounded. Using this compactness again, we see that there is a $\delta>0$ such that $V\coloneqq\{x\in U: \text{dist}(x,\partial U)>\delta\}\supset A$. Here, we are ensuring that $\partial V$ does not coincide with $\partial\Omega$. $V$ may be written as a union of open balls with radii less than $\varepsilon/2$ and, since this is an open covering of $A$, we may find $N$ of these balls, $B(x_1,s_1),...,B(x_N,s_N)$, that still cover $A$. Now, we could slightly enlarge the radii of these balls and still have a covering of $A$. This is because their boundaries have distance greater than $\delta$ from $\partial\Omega$. Enlarging the radii will increase the total volume of the balls, but there is a value $R\in(0,\min\{\delta,\varepsilon/2\})$ whereby increasing the radii by no more than $R$ ensures that the balls remain inside $U$. For each $i$, we consider $\{B(x_i,t)\}_{t\in(s_i,R)}$. Using {\it foliations by Borel sets} (Section \ref{measures}), we may find some $r_i\in(s_i,R)$ such that $|\nabla g|(\partial B(x_i,r_i))=0$. $R$ was chosen so that $\cup_{i=1}^N B(x_i,r_i)\subset U$, so monotonicity, we have that $|\nabla g|(\cup_{i=1}^N B(x_i,r_i))<\varepsilon$
    and we have now proved our claim.

    Let $\phi\in C_c(\cup_{i=1}^N B(x_i,r_i))$. Then
    \begin{align*}
        \overline{\F\times\nabla g} (\phi) &= \lim_{k\to\infty} \int_\Omega\phi\F\times\nabla g_k
        \leq ||\phi||_\infty||\F||_\infty \limsup_{k\to\infty}|\nabla g_k|(\cup_{i=1}^N B(x_i,r_i))\\
        &= ||\phi||_\infty||\F||_\infty |\nabla g|(\cup_{i=1}^N B(x_i,r_i))
        \leq\varepsilon ||\phi||_\infty||\F||_\infty
    \end{align*}
    where the last equality follows by Proposition \ref{weak* results}.1 and the fact that  for all $i$,\newline $|\nabla g|(\partial B(x_i,r_i))=0$ and hence, $|\nabla g|\left(\partial(\cup_{i=1}^N B(x_i,r_i))\right)=0$. Taking $\varepsilon\to 0$ gives our result.
\end{proof}

\begin{theorem}
    Let $\bm{F}\in\mathcal{CM}^\infty(\Omega)$ and $E\Subset\Omega$ a set of finite perimeter. Then we have
    \begin{align}
        \Curl(\chi_E\bm{F}) &= \chi_{E^1}\Curl(\bm{F}) - 2\overline{\chi_E\bm{F}\times\nabla\chi_E} \label{p1}\\
        \Curl(\chi_E\F) &= \chi_{E^1\cup\rb E}\Curl\F-2\overline{\chi_{\Omega\setminus E}\F\times\nabla\chi_E}\label{p2}\\
        \chi_{\rb E}\Curl\F &= 2\overline{\chi_{\Omega\setminus E}\F\times\nabla\chi_E} - 2\overline{\chi_{E}\F\times\nabla\chi_E}\label{p3}
    \end{align}
\end{theorem}
\begin{proof}
    This is just repeated application of Theorem \ref{ProductRule} and using the values of $\chi_E^*$ outside of a $\h^2$-null set, as discussed in Section \ref{finite perimeter}. We have that
    \begin{equation}\label{product1}
         \Curl(\chi_E^2\F) = \Curl(\chi_E\F) = \chi_E^*\Curl\F-\overline{\F\times\nabla\chi_E}.
    \end{equation}
    We also have that
    \begin{align*}
        \Curl(\chi_E^2\F) &= \Curl(\chi_E\chi_E\F)
        =\chi_E^*\Curl(\chi_E\F)-\overline{\chi_E\F\times\nabla\chi_E}\\
        &=\chi_E^*\left(\chi_E^*\Curl\F-\overline{\F\times\nabla\chi_E}\right)-\overline{\chi_E\F\times\nabla\chi_E}\\
        &= \left(\chi^*_E\right)^2\Curl\F -\chi_E^*\overline{\F\times\nabla\chi_E}-\overline{\chi_E\F\times\nabla\chi_E}.
    \end{align*}
    Combining this with \eqref{product1} gives
    \[
    \left(\chi^*_E\right)^2\Curl\F -\chi_E^*\overline{\F\times\nabla\chi_E}-\overline{\chi_E\F\times\nabla\chi_E} = \chi_E^*\Curl\F-\overline{\F\times\nabla\chi_E}
    \]
    which simplifies to
    \[
    \chi^*_E\left(1-\chi^*_E\right)\Curl\F = \left(1-\chi^*_E\right)\overline{\F\times\nabla\chi_E}-\overline{\chi_E\F\times\nabla\chi_E}.
    \]
    So, using the known values of $\chi_E^*$,
    \begin{equation}\label{product1.1}
        \frac{1}{4}\chi_{\rb E}\Curl\F = \left(1-\chi^*_E\right)\overline{\F\times\nabla\chi_E}-\overline{\chi_E\F\times\nabla\chi_E}.
    \end{equation}
    Since $|\overline{\F\times\chi_E}|\ll|\nabla\chi_E|$ and $|\nabla\chi_E|$ is concentrated on $\rb E$ we have that
    \[
    \chi_E^*\overline{\F\times\nabla\chi_E} =\chi_E^*|_{\rb E}\overline{\F\times\nabla\chi_E} = \frac{1}{2}\chi_{\rb E}\overline{\F\times\nabla\chi_E} 
    \]
    and
    \[
    \overline{\F\times\nabla\chi_E} =\chi_{\rb E}\overline{\F\times\nabla\chi_E}
    \]
    Thus, \eqref{product1.1} becomes
    \[
    \frac{1}{4}\chi_{\rb E}\Curl\F = \frac{1}{2}\overline{\F\times\nabla\chi_E}-\overline{\chi_E\F\times\nabla\chi_E}.
    \]
    That is
    \begin{equation}\label{product2}
    \frac{1}{2}\chi_{\rb E}\Curl\F = \overline{\F\times\nabla\chi_E}-2\overline{\chi_E\F\times\nabla\chi_E}.    
    \end{equation}
    Using bilinearity of the cross product and linearity of integration we have that
    \begin{equation}\label{product3}
        \overline{\F\times\nabla\chi_E} = \overline{\left(\chi_E+\chi_{\Omega\setminus E}\right)\F\times\nabla\chi_E} = \overline{\chi_E\F\times\nabla\chi_E} + \overline{\chi_{\Omega\setminus E}\F\times\nabla\chi_E}.
    \end{equation}
    So, combining \eqref{product2} and \eqref{product3} gives \eqref{p3}. 
    
    \eqref{p1} follows from
    \begin{align*}
        \Curl(\chi_E\F) = \chi_E^*\Curl\F-\overline{\F\times\nabla\chi_E}
        &= \chi_{E^1}\Curl\F + \frac{1}{2}\chi_{\rb E}\Curl\F - \overline{\F\times\nabla\chi_E}\\
        &= \chi_{E^1}\Curl\F -2\overline{\chi_E\F\times\nabla\chi_E},
    \end{align*}
    where the first equality is \eqref{product1} and the last equality follows from \eqref{product2}. \eqref{p2} follows from
    \begin{align*}
        \Curl(\chi_E\F) &= \chi_E^*\Curl\F-\overline{\F\times\nabla\chi_E}\\
        &= \chi_E^*\Curl\F + \frac{1}{2}\chi_{\rb E}\Curl\F - \frac{1}{2}\chi_{\rb E}\Curl\F - \overline{\F\times\nabla\chi_E}\\
        &\begin{aligned} 
        =\chi_{E^1\cup\rb E}\Curl\F &-\left( \overline{\chi_{\Omega\setminus E}\F\times\nabla\chi_E} - \overline{\chi_{E}\F\times\nabla\chi_E} \right)\\
        &- \left( \overline{\chi_E\F\times\nabla\chi_E} + \overline{\chi_{\Omega\setminus E}\F\times\nabla\chi_E} \right) \end{aligned}\\
        &= \chi_{E^1\cup\rb E}\Curl\F - 2 \overline{\chi_{\Omega\setminus E}\F\times\nabla\chi_E},
    \end{align*}
    where the first equality is \eqref{product1} and the penultimate equality follows from from \eqref{p3} and \eqref{product3}
\end{proof}
    Since \ref{ProductRule} tells us that $2\overline{\chi_E\F\times\nabla\chi_E}$ and $2\overline{\chi_{\Omega\setminus E}\F\times\nabla\chi_E}$ are absolutely continuous with respect to $|\nabla\chi_E|$, we may consider their respective Radon-Nikodym derivatives, $\mF_i\times\n_E$ and $\mF_e\times\n_E$. We may re-write the previous theorem in terms of these functions:
\begin{theorem}\label{three formula 2}
      Let $\bm{F}\in\mathcal{CM}^\infty(\Omega)$ and $E\Subset\Omega$ a set of finite perimeter. Then we have
    \begin{align}
        \Curl(\chi_E\bm{F}) &= \chi_{E^1}\Curl(\bm{F}) - \mF_i\times\n_E|\nabla\chi_E|,\label{f1}\\
        \Curl(\chi_E\F) &= \chi_{E^1\cup\rb E}\Curl\F-\mF_e\times\n_E|\nabla\chi_E|,\label{f2}\\
        \chi_{\rb E}\Curl\F &= (\mF_e\times\n_E-\mF_i\times\n_E)|\nabla\chi_E|.\label{f3}
    \end{align}
\end{theorem}
We note that since $\Curl\F\ll\h^2$ and $|\nabla\chi_E|=\h^2\res\rb E$, there exists $\bm{q_2}\in L^1(\h^2,\Omega)$ such that $\h^2$-a.e. on $\rb E$,
\[
\bm{q_2}=\mF_e\times\n_E-\mF_i\times\n_E,
\]
Here, we can see that the jump of the tangential traces (discussed in Section \ref{tangential traces section}) is determined by the measure $\Curl\F\res\rb E$ and, in particular, its Radon-Nikodym derivative with respect to $\h^2\res\rb E$, namely $\bm{q_2}$.
\begin{proposition}\label{curl vanishes}
    If $\F\in\mathcal{CM}^\infty(\Omega)$ has compact support in $\Omega$, then $\Curl\F (\Omega)=\bm{0}$.
\end{proposition}
\begin{proof}
    Since $\F$ has compact support, there exists an open set $U$ such that $\spt(\F)\subseteq U \Subset\Omega$. From the definition of $|\Curl\F|$, this means that $\Curl\F$ will return zero when measuring any subsets of $\Omega\setminus\overline{U}$. Choose $\phi\in C^\infty_c$ such that $\phi\equiv 1$ on a neighborhood of $U$. Then
    \[
    \bm{0} = -\int_{\Omega\setminus U}\F\times\nabla\phi = - \int_{\Omega}\F\times\nabla\phi = \int_\Omega \phi \dif\, (\Curl\F) = \int_{\overline{U}}\phi \dif\, (\Curl\F) = \Curl\F(\overline{U})
    \]
    Hence, $\Curl\F(\Omega)=\bm{0}$.
\end{proof}
We now prove our trace theorem, which shows us that $(\mF_i\times\n_E)$ and $(\mF_e\times\n_E)$ represent the internal and external tangential traces, respectively.
\begin{theorem}
\label{first trace theorem}
    Let $\F\in\mathcal{CM}^\infty(\Omega)$ and $E\Subset\Omega$ a set of finite perimeter. Then for any $\varphi\in\Lip_c(\Omega)$,
    \begin{equation}
    \label{uno}
        \int_{\rb E}\varphi (\mF_i\times\n_E)\dif\h^2 = \int_{E^1}\varphi \dif\, (\Curl\F) - \int_{E}\F\times\nabla\varphi\dif\lb^3
    \end{equation}
    and 
    \begin{equation}
    \label{dos}
        \int_{\rb E}\varphi (\mF_e\times\n_E)\dif\h^2 = \int_{E^1\cup\rb E}\varphi \dif\, (\Curl\F) - \int_{E}\F\times\nabla\varphi\dif\lb^3.
    \end{equation}
\end{theorem}

\begin{proof}
    Consider some $\varphi\in\Lip_c(\Omega)$. There exists an open set $U\Subset\Omega$ such that $\spt(\chi_E\varphi)\Subset U$. Applying the product rule \ref{ProductRule} to $\varphi\chi_E\F$ gives
    \begin{align*}
        \Curl(\varphi\chi_E\F) &= \varphi^*\Curl(\chi_E\F) - \overline{\chi_E\F\times\nabla\varphi}\\
        & =\varphi\left(\chi_{E^1}\Curl(\bm{F}) - \mF_i\times\n_E|\nabla\chi_E|\right) - \chi_E\F\times\nabla\varphi\lb^3, 
    \end{align*}
    where the last equality follows from \eqref{f1} and continuity of $\varphi$.
    Now, evaluate this over $U$ and apply \ref{curl vanishes} to obtain the first desired expression.
    
    Similarly, applying the product rule \ref{ProductRule} to $\varphi\chi_E\F$ and using \eqref{f2} gives
    \begin{equation*}
        \Curl(\varphi\chi_E\F) = \varphi(\chi_{E^1\cup\rb E}\Curl\F-\mF_e\times\n_E|\nabla\chi_E|) - \chi_E\F\times\nabla\varphi\lb^3.
    \end{equation*}
    and we get the second desired expression in the same way.
\end{proof}

Here, we remark upon the version of this Gauss-Green theorem where $\F$ is just Lipschitz continuous. Federer and De Giorgi's work (\cite{Fed45},\cite{Fed58}\cite{DeGiorgi1},\cite{DeGiorgi2}) essentially show this, although they consider divergence instead of curl.
\begin{theorem}\label{Lipschitz case}
    If $E$ is a set of finite perimeter and $\F:\R^3\to\R^3$ is Lipschitz, then
    \[
    \int_E(\Curl\F)(x) \dif\lb^3(x) = \int_{\rb E}\F(y)\times\n_E(y)\dif\h^2(y).
    \]
\end{theorem}

\section{Weak Limit Approach}\label{functional method section}
In this section, we adapt \v Silhav\'y's approach in \cite{Silhavy05} for divergence measure fields to curl-measure fields, although we remove some assumptions. Rather than present the result as one single theorem, we first show that the interior tangential trace is represented by a {\it measure} and then show that it is represented by a {\it function}. This is similar to the approach taken in \cite{Silhavy05} and \cite{Silhavy23} and it highlights the construction of a function $\bm{q}_0$ which will be used in Section \ref{tangential property section}. Note that the trace theorems in this section do not consider the exterior tangential trace because we have already obtained the relevant Gauss-Green formula in Section \ref{product method section} and it is not strictly required when proving the tangential property (Theorem \ref{maintheorem}). 
\begin{proposition}\label{tangential measure}
     Let $\bm{F}\in\mathcal{CM}^\infty(\Omega)$, $E\Subset\Omega\subset\R^3$ a bounded set of finite perimeter.
     Then the tangential trace is represented by a measure $\tau^E=\pi^E-\sigma^E$ where
     \[
     \pi^E=\bm{q_0}\h^2\res\partial^*E
     \]
     and $\bm{q_0}\in L^1(\partial^*E;\R^3,\h^2)$ is the weak* limit of $\left\{\bm{F}_\varepsilon|_{\partial^*E}\times\bm{\nu}_E\right\}_{\varepsilon>0}$ and $\sigma^E$ is the weak* limit of $\left\{\sigma_\varepsilon\right\}_{\varepsilon>0}$ defined by
     \[
     \langle\sigma_\varepsilon, \varphi\rangle = \int_{\partial^*E}\int_E\varphi(y)\rho_\varepsilon(x-y)\dif\lb^3(y) \dif\,(\Curl\bm{F})(x)
     \]
     for all $\varphi\in C_c(\Omega)$.
     That is, for all $\varphi\in\Lip_c(\Omega)$,
     \[
     \int_{\Omega}\varphi \dif\tau^E  = \int_{E^1}\varphi \dif\, (\Curl\F) - \int_{E}\F\times\nabla\varphi.
     \]
     The weak* convergence here is understood to be along a particular sequence that converges to $0$.
\end{proposition}
\begin{proof}
    To simplify notation, for $\varepsilon>0$, we will denote $\bm{F}_\varepsilon|_{\partial^*E}\times\bm{\nu}_E$ by $\bm{F}_\varepsilon\times\bm{\nu}_E$.
    We first consider $\bm{F}_\varepsilon\times\bm{\nu}_E\h^2\res\partial^*E$. To see that this is a $\R^3$-valued Radon measure on $\Omega$, we may consider $\bm{F}_\varepsilon\times\bm{\nu}_E$ component-wise and use the discussion in Section \ref{measures}.
    Since $||\F_\varepsilon||_{L^\infty(\Omega^\varepsilon)}\leq ||\F||_{L^\infty(\Omega)}||\rho_\varepsilon||_{L^1(\Omega^\varepsilon)}\leq ||\F||_{L^\infty(\Omega)}$, we have that 
    \[
    \left|\F_\varepsilon\times\bm{\nu}_E\h^2\res\partial^*E\right| (\Omega) =
    \int_{\partial^*E}|\bm{F}_\varepsilon\times\bm{\nu}_E|\dif\h^2\leq
    ||\F||_{L^\infty(\Omega)}\h^2(\rb E).
    \]
    This means that $\left\{\left(\bm{F}_\varepsilon\times\bm{\nu}_E\h^2\res\partial^*E\right) (\Omega)\right\}_{\varepsilon>0}$ is uniformly bounded and therefore, Proposition \ref{weak* results}.2 gives us that for any positive sequence, $\{\varepsilon_j\}_{j\in\N}$, which converges to $0$, there exists a subsequence of $\F_{\varepsilon_j}\times\bm{\nu}_E\h^2\res\partial^*E$ with a weak* limit. We call this $\R^3$-valued Radon measure on $\Omega$, $\pi^E$ and we will refer to it as the weak* limit of $\F_\varepsilon\times\bm{\nu}_E\h^2\res\partial^*E$ as $\varepsilon\to 0^+$.
    
    Since $|\Curl\F|$ is a Radon measure and $\rb E$ is a subset of the compact set, $\partial E$, we have that for all $\varphi\in C_c(\Omega)$, 
    \[
    |\langle\sigma_\varepsilon,\varphi\rangle|\leq||\varphi||_{L^\infty(\Omega)} \lb^3(E)|\Curl\bm{F}|(\partial^*E)<\infty,
    \]
    which implies that $\left\{|\sigma_\varepsilon|(\Omega)\right\}_{\varepsilon>0}$ is uniformly bounded. Thus, similarly, we see that $\sigma_\varepsilon$ has a weak* limit as $\varepsilon\to 0^+$ and we call this $\R^3$-valued Radon measure on $\Omega$, $\sigma^E$. As before, we are suppressing the notation involving a subsequence.
    Note that this is {\it not} the standard $\varepsilon$-regularization of a vector-valued Radon measure as seen in Proposition \ref{weak* results}.3.
    
    Consider some $\varphi\in\Lip_c(\Omega)$. Since $\bm{F}_\varepsilon$ is smooth, we have from Theorem \ref{Lipschitz case} that
    \begin{equation}\label{smooth curl theorem}
        \int_{\partial^*E}\varphi\bm{F}_\varepsilon\times\bm{\nu}_E\dif\h^2 = \int_{E}\varphi\Curl\bm{F}_\varepsilon\dif\lb^3-\int_{E}\bm{F}_\varepsilon\times\nabla\varphi \dif\lb^3.
    \end{equation}  
    The term on the left hand side of \eqref{smooth curl theorem} converges to 
    \[
    \int_{\partial^*E}\varphi \dif\pi^E,
    \]
    as $\varepsilon\to 0^+$, by the definition of weak* convergence.
    The second term on the right hand side of \eqref{smooth curl theorem} converges to 
    \[
    \int_{E}\bm{F}\times\nabla\varphi \dif\lb^3,
    \]
    as $\varepsilon\to 0^+$, if we consider the difference of $\int_{E}\bm{F}\times\nabla\varphi \dif\lb^3$ and $\int_{E}\bm{F}_\varepsilon\times\nabla\varphi \dif\lb^3$ and use that $\varphi$ has compact support and $\bm{F}_\varepsilon\to\bm{F}$ in $L^1_{\text{loc}}(\Omega)$.
    
    For the first term on the right hand side of \eqref{smooth curl theorem}, using that $(\Curl\F)_\varepsilon=\Curl(\F_\varepsilon)$ (as seen in the proof of Theorem \ref{ProductRule}), we note that
    \begin{align*}
         \int_{E}\varphi(y)\Curl\bm{F}_\varepsilon (y) \dif\lb^3(y) &= \int_{E}\varphi(y)\int_\Omega\rho_\varepsilon(x-y)\dif\,(\Curl\F)(x) \dif\lb^3 (y)\\ 
         &= \int_\Omega\int_{E}\varphi(y)\rho_\varepsilon(x-y) \dif\lb^3 (y) \dif\,(\Curl\F)(x)\\
         &= \int_\Omega\varphi_\varepsilon(x) \dif\,(\Curl\F)(x),
    \end{align*}
    where, in the last equality, we are abusing notation slightly by defining $\varphi_\varepsilon$ by
    \begin{equation}\label{one sided mollification}
        \varphi_\varepsilon(x) \coloneqq \int_E\varphi(y)\rho_\varepsilon(x-y)\dif\lb^3(y)\qquad\text{for } x\in\Omega^\varepsilon.
    \end{equation}
    This is not standard notation; usually $\varphi_\varepsilon$ denotes mollification over $\Omega^\varepsilon$ for some $\varepsilon>0$. We use it here so that the mollification only takes information coming from $E$.
    Consider that
    \begin{equation}\label{split integral}
        \int_\Omega\varphi_\varepsilon \dif\,(\Curl\F) = \int_{E^1} \varphi_\varepsilon \dif\,(\Curl\F) + \int_{E^0} \varphi_\varepsilon \dif\,(\Curl\F) + \int_{\partial^*E} \varphi_\varepsilon \dif\,(\Curl\F).
    \end{equation}
    This is because $\Omega$ is the disjoint union of $E^1$, $E^0$, $\rb E$ and some $\h^2$-null set, as well as the fact that $\Curl\F\ll\h^2$.
    By the following argument, the three terms on the right hand side of \eqref{split integral} converge, respectively, to
    \[
    \int_{E^1} \varphi \dif\,(\Curl\F), \qquad 0,\qquad \text{ and } \qquad\int_\Omega\varphi \dif\sigma^E.
    \]
    The convergence of the first two limits follow from the dominated convergence theorem and that, outside of a $\h^2$-null set,
    \begin{equation}\label{precise rep corollary}
    \varphi_\varepsilon(x)\longrightarrow
    \begin{cases}
        \varphi(x),\quad&x\in E^1\,,\\
        \frac{1}{2}\varphi(x),\quad&x\in \rb E\,,\\
        0,\quad&x\in E^0.\,
    \end{cases}        
    \end{equation}
    \eqref{precise rep corollary} may be seen by first considering that if $x\in E^1\cup\rb E$, then
    \begin{align*}
         \lim\limits_{\varepsilon\to 0^+}\left|\int_E(\varphi(y)-\varphi(x))\rho_\varepsilon(x-y)\dif\lb^3(y)\right|&\leq \lim\limits_{\varepsilon\to 0^+}\frac{1}{\varepsilon^3}\int_E\left|\varphi(y)-\varphi(x)\right|\rho\left(\frac{x-y}{\varepsilon}\right)\dif\lb^3(y)\\
         &\leq \lim\limits_{\varepsilon\to 0^+}\frac{1}{\varepsilon^3}\int_{E\cap B(x,\varepsilon)}\left|\varphi(y)-\varphi(x)\right|\dif\lb^3(y) =0,
    \end{align*}
    where we have used the Lebesgue differentiation theorem (Theorem 3.21 in \cite{Folland}) and the continuity of $\varphi$. We note that we may apply this theorem since, due to the definition of $E^1$ and \eqref{boundary has half density}, for all sufficiently small $\varepsilon>0$, there exists $\alpha>0$, independent of $\varepsilon$, such that $|E\cap B(x,\varepsilon)|>\alpha|B(x,\varepsilon)|$.
   Thus,
   \begin{align*}
       \lim\limits_{\varepsilon\to 0^+}\int_E\varphi(y)\rho_\varepsilon(x-y)\dif\lb^3(y) &= \lim\limits_{\varepsilon\to 0^+}\int_E\varphi(x)\rho_\varepsilon(x-y)\dif\lb^3(y)\\ 
       &= \varphi(x)\lim\limits_{\varepsilon\to 0^+}\int_E\rho_\varepsilon(x-y)\dif\lb^3(y)
   \end{align*}
    and the first two cases in \eqref{precise rep corollary} then follow by considering the precise representation of $\chi_E$ discussed in Section \ref{finite perimeter}. The third case in \eqref{precise rep corollary} follows by considering that if $x\in E^0$, as $\varepsilon\to 0^+$,
    \[
    \left|\int_E\varphi(y)\rho_\varepsilon(x-y)\dif\lb^3(y)\right|\leq \frac{||\varphi||_\infty}{\varepsilon^3}\int_{E\cap B(x,\varepsilon)}\rho\left(\frac{x-y}{\varepsilon}\right)\dif\lb^3(y)\leq \frac{||\varphi||_\infty|E\cap B(x,r)|}{\varepsilon^3}\to0.
    \]
    The convergence of the third term in \eqref{split integral} to $\int_\Omega\varphi \dif\sigma^E$  follows from the definition of a weak* limit and the fact that, by definition of $\sigma_\varepsilon$,
    \[
    \int_{\partial^*E} \varphi_\varepsilon \dif\,(\Curl\F) = \langle \sigma_\varepsilon, \varphi \rangle.
    \]
    Applying all of these limits that we have found to \eqref{smooth curl theorem} gives us that
    \[
    \int_{\partial^*E}\varphi \dif\pi^E = \int_{E^1} \varphi \dif\,(\Curl\F) + \int_\Omega\varphi \dif\sigma^E  - \int_{E}\bm{F}\times\nabla\varphi \dif\lb^3. 
    \]
    Therefore, the tangential trace is indeed represented by the measure $\tau^E = \pi^E - \sigma^E$.
    
    For all $\varepsilon>0$, and for all $x\in\partial^*E$,
    \begin{equation*}
        \big|(\bm{F}_\varepsilon\times\bm{\nu}_E)(x)\big| \leq \big|\bm{F}_\varepsilon(x)\big| \cdot\big|\n_E(x)\big|\leq||\F||_{L^\infty(\Omega)}.
    \end{equation*}  
    This tells us that $\{\bm{F}_\varepsilon\times\bm{\nu}_E\}_{\varepsilon>0}$ is uniformly bounded and so we have a subsequence that converges weakly in $L^\infty(\rb E;\R^3,\h^2)$ i.e. there exists $\bm{q}_0\in L^\infty(\partial^*E;\R^3,\h^2)$ such that as $\varepsilon\to 0^+$,
    \[
    \int_{\partial^*E}\bm{F}_\varepsilon\times\bm{\nu}_E \, g \dif\h^2 \longrightarrow \int_{\partial^*E}\bm{q}_0 \, g \dif\h^2,
    \]
    for any $g\in L^1(\rb E,\h^2)$. This follows from our discussion Section \ref{functional analysis}, considering $\F_\varepsilon|_{\rb E}\times\n_E$ component-wise and noting that $(\partial^*E,\h^2)$ is a finite measure space and hence, $\sigma$-finite. As before, we are suppressing the notation involving a positive sequence converging to $0$. $ L^\infty(\partial^*E;\R^3,\h^2) \subseteq L^1(\partial^*E;\R^3,\h^2)$, since we are considering Radon measures restricted to compactly contained sets here. Thus, in particular, we may say that there exists $\bm{q}_0\in L^1(\partial^*E;\R^3,\h^2)$ such that  as $\varepsilon\to 0^+$,
    \[
    \int_{\partial^*E}\bm{F}_\varepsilon\times\bm{\nu}_E \, g \dif\h^2 \longrightarrow \int_{\partial^*E}\bm{q}_0 \, g \dif\h^2,
    \]
    for any $g\in L^1(\rb E,\h^2)$. That is 
    \[
        \bm{F}_\varepsilon\times\bm{\nu}_E \rightharpoonup \bm{q}_0 \qquad \text{in }\quad L^\infty(\rb E;\R^3,\h^2).
    \]

    Let $\pi_\varepsilon\coloneqq\bm{F}_\varepsilon\times\bm{\nu}_E\h^2\res\rb E$. We know that for any $g\in L^1(\rb E,\h^2)$, as $\varepsilon\to 0^+$,
    \[
    \int_{\Omega}g \, \dif\pi_\varepsilon \longrightarrow \int_{\Omega} g \, \bm{q}_0 \dif\h^2\res\rb E,
    \]
    appealing to our discussion in Section \ref{measures}. In fact, we may take $g$ from $C_c(\Omega)$, since restricting to $\rb E$ produces an element of $L^\infty(\rb E,\h^2)$. Consequently, we have that $\pi_\varepsilon$ weak* converges to $\bm{q}_0\h^2\res\rb E$. By uniqueness of weak* limits, we must have that 
    \[
    \bm{q}_0\h^2\res\rb E =\pi^E.
    \]
\end{proof}
We remark here that $\pi^E$ is constructed using convolutions across {\it both sides} of $\rb E$. The interior tangential trace should only take information from $E$, so it follows that we must subtract a term from $\pi^E$ in order to obtain the interior tangential trace. The reason that we subtract $\sigma^E$ in particular is because, as we will note in the proof of the next theorem, $\sigma^E=\frac{1}{2}\bm{q_2}\h^2\res\rb E$. In the discussion below Theorem \ref{three formula 2}, we saw how this term determined the jump of the tangential traces. The proofs provide explicit reasoning, but the intuition is as follows. Since $\pi^E$ determines an averaging of the tangential component of $\F$ along $\rb E$ from both sides, we must subtract $\sigma^E$, which determines the jump of the tangential component across $\rb E$, in order to obtain the interior tangential trace. Let us observe this in the context of an explicit example. Due to the particularly simple nature of Example \ref{example1}, we see that for all $\varepsilon>0$, the terms $\F_\varepsilon\times\n_E$ are exactly the average of $\F_1\times\n_E$ and $\F_2\times\n_E$. Trivially then, $\bm{q}_0=\frac{1}{2}\left(\F_1\times\n_E+\F_2\times\n_E\right)$. As we have alluded to and as the pointwise defined formula in the next theorem makes clear, $\mF_i\times\n_E$ and $\F_e\times\n_E$ are approximations of $\F\times\n_E$ from $E$ and $\Omega\setminus E$, respectively. It then follows that for Example \ref{example1}, $\mF_i\times\n_E=\F_1\times\n_E$, $\F_e\times\n_E=\F_2\times\n_E$ and
\[
\mF_i\times\n_E = \frac{1}{2}\left(\mF_i\times\n_E+\F_e\times\n_E\right)-\frac{1}{2}\left(\F_e\times\n_E-\mF_i\times\n_E\right) = \bm{q}_0 - \frac{1}{2}\bm{q}_2
\]
which immediately implies that
\[
\left(\mF_i\times\n_E\right)\h^2\res\rb E = \pi^E-\sigma^E.
\]
 
We now obtain our trace theorem once again. We provide only the Gauss-Green formula for the interior tangential trace, for brevity, but this time we give an explicit pointwise defined formula for the interior tangential trace. We will see that the function $\bm{q}^E$ in this theorem, agrees $\h^2$-a.e. with $\mF_i\times\n_E$ from Theorem \ref{first trace theorem}. However, due to the independence of the proofs of the two theorems, we label them as such.
\begin{theorem}\label{Second trace theorem}
    Let $\bm{F}\in\mathcal{CM}^\infty(\Omega)$, $E\Subset\Omega\subset\R^3$ a set of finite perimeter.
    Then there exists $\bm{q}^E\in L^\infty(\rb E;\R^3,\h^2)$ such that, for all $\varphi\in\Lip_c(\Omega)$,
    \begin{equation}\label{curl as function equation}
 \int_{\partial^*E}\varphi\bm{q}^E\dif\h^2 = \int_{E^1}\varphi\dif\,(\Curl\bm{F})-\int_{E}\bm{F}\times\nabla\varphi \dif\lb^3.
    \end{equation}
    Moreover, a representative for $\bm{q}^E$ can be defined as follows. For $x\in\rb E$,
    \begin{equation}\label{trace is a function expression}
        \bm{q}^E(x)= \begin{cases}
            \lim\limits_{r\to 0^+}\frac{3}{\omega_2r^3}\int\limits_{B(x, \n_E, r)} \bm{F}(y)\times\frac{y-x}{|y-x|}\dif\lb^3(y), \quad &\text{ if this limit exists and is finite},\\
            \bm{0}, \quad &\text{otherwise},
        \end{cases}
    \end{equation}
    where $B(x, \n_E, r)\coloneqq B(x, r)\cap\left\{ y\in\Omega : (y-x)\cdot\n_E(x)>0 \right\}$. The limit in \eqref{trace is a function expression} exists and is finite $\h^2$-a.e.
\end{theorem}
\begin{proof}
    Our hypothesis here includes that of Proposition \ref{tangential measure}, so its conclusions are available to us.
    We recall that $\Curl\bm{F}\ll\h^2$ and hence, the Radon-Nikodym differentiation theorem tells us that there exists $\bm{q_2}\in L^1(\rb E, \h^2)$ such that
    \[
    \Curl\bm{F}\res\rb E = \bm{q_2} \h^2\res\rb E.
    \]
    Recalling the definition of $\sigma^E$, we see that for any $\varphi\in\Lip_c(\Omega)$,
    \begin{align*}
        \langle\sigma^E,\varphi\rangle = \lim_{\varepsilon\to 0^+} \int_{\partial^*E}\int_E\varphi(y)\rho_\varepsilon(x-y)\dif\lb^3(y) \dif\,(\Curl\F)(x)
        &= \lim_{\varepsilon\to 0^+} \int_{\partial^*E} \varphi_\varepsilon \dif\,(\Curl\F)\\
        &= \lim_{\varepsilon\to 0^+} \int_{\partial^*E} \varphi_\varepsilon \bm{q_2} \dif\h^2.
    \end{align*}
    Note that we are still using the definition of $\varphi_\varepsilon$ as given by \eqref{one sided mollification}. Applying \eqref{precise rep corollary}, which says that 
    \begin{equation*}
        \varphi_\varepsilon(x)\to\frac{1}{2}\varphi(x) \qquad\text{ for } \h^2\text{-a.e.}\, x\in\rb E,
    \end{equation*}
    we obtain from the dominated convergence theorem that
    \[
    \langle\sigma^E,\varphi\rangle = \frac{1}{2}\int_{\rb E} \varphi \bm{q_2} \dif\h^2.
    \]
    Thus, $\sigma^E=\frac{1}{2} \bm{q_2} \h^2\res\rb E$ and if we let 
    \begin{equation}\label{q0 defn}
        \bm{q}^E\coloneqq\bm{q_0}-\frac{1}{2} \bm{q_2},
    \end{equation}
    then \eqref{curl as function equation} holds. Explicitly, we see from Proposition \ref{tangential measure} that since $\pi^E=\bm{q}_0\h^2\res\rb E$, we have that
    \begin{align*}
        \int_{\partial^*E}\varphi\bm{q}^E\dif\h^2 = \int_{\rb E} \varphi\bm{q_0} \dif\h^2 - \frac{1}{2}\int_{\rb E} \varphi\bm{q_2} \dif\h^2
        &=\langle\pi^E,\varphi\rangle - \langle\sigma^E,\varphi\rangle
        =\langle\tau^E,\varphi\rangle\\
        &= \int_{E^1}\varphi \dif\,(\Curl\F)-\int_{E}\bm{F}\times\nabla\varphi \dif\lb^3.
    \end{align*}
    To prove \eqref{trace is a function expression}, we begin by first showing that for $\h^2$-a.e. $x\in\rb E$,
    \begin{equation}\label{trace is a function intermediary formula}
        \bm{q}^E (x) = \lim\limits_{r\to 0^+}\frac{3}{\omega_2r^3}\int\limits_{E\cap B(x, r)} \bm{F}(y)\times\frac{y-x}{|y-x|}\dif\lb^3(y).
    \end{equation}
    By a Lebesgue differentiation theorem generalized to Radon measures (see Theorem 5.16 in \cite{Maggi}, for example), for $\h^2$-a.e. $x\in\rb E$,
    \begin{align*}
        \bm{q}^E(x) &= \lim\limits_{r\to 0^+}\frac{\int\limits_{B(x,r)}\bm{q}^E\dif\h^2\res\rb E}{(\h^2\res\rb E)(B(x,r))}\\
        &= \lim\limits_{r\to 0^+}\left(\frac{\int\limits_{B(x,r)}\bm{q}^E\dif\h^2\res\rb E}{(\h^2\res\rb E)(B(x,r))} \right)  \lim\limits_{r\to 0^+} \left(\frac{(\h^2\res\rb E)(B(x,r))}{\omega_2r^2}\right)\\
        &=  \lim\limits_{r\to 0^+}\frac{\int\limits_{B(x,r)}\bm{q}^E\dif\h^2\res\rb E}{\omega_2r^2}
    \end{align*}
    where the second equality follows from \eqref{rectifiability in terms of measure}. Hence, for $\h^2$-a.e. $x\in\rb E$,
    \begin{equation}\label{first limit}
        \bm{q}^E(x) = \lim\limits_{r\to 0^+}\frac{1}{\omega_2r^2}\int\limits_{\rb E \cap B(x,r)}\bm{q}^E\dif\h^2.
    \end{equation}
    The upper density theorem (Theorem 3.8 in \cite{Simon}) states that if $\mu$ is a Borel regular outer measure on a metric space $X$, $A$ is a $\mu$-measurable subset of $X$ such that $\mu(A)<\infty$ and $m\geq 0$, then for $\h^m$-a.e. $x\in X\setminus A$,
    \[
    \limsup\limits_{r\to 0^+}\frac{\mu(A\cap B(x,r))}{\omega_m\rho^m}=0.
    \]
    Applying this to the case where $X=\Omega$, $\mu=|\Curl\F|$, $A=E^1$ and $m=2$ gives us that for $\h^2$-a.e. $x\in\rb E$,
    \begin{equation}\label{second limit}
        \lim\limits_{r\to 0^+}\frac{1}{\omega_2r^2}|\Curl\bm{F}|(E^1\cap B(x,r)) = 0.
    \end{equation}
    Here, we are using the fact that $|\Curl\F|(E^1)<\infty$ which follows since $E^1\subset\overline{E}$, $\overline{E}$ is compact and $|\Curl\F|$ is a Radon measure.
    Let $x\in\rb E$ be a point such that both \eqref{first limit} and \eqref{second limit} hold and for any $r>0$, let 
    \[
    E_r\coloneqq E^1\cap B(x,r).
    \]
    We note that since $|E\Delta E^1|=0$, we have that
    \begin{equation}\label{Er equiv}
        |E_r\Delta(E\cap B(x,r))|=0.
    \end{equation}
    Define $\varphi_r\in\Lip_c(\Omega)$ by
    \[
    \qquad\varphi_r(y) = \max\{r-|y-x|, 0\,\}
    \]
    for any $y\in\Omega$. Consider the spheres of radius $s$ centered at $x$, $\partial B(x,s)$. $\varphi_r$ is radial about $x$ and, as we increase $s$ from $0$ to $r$, the values of $\varphi_r$ on $\partial B(x,s)$ decrease from $r$ to $0$. The function remains $0$ on all spheres of radius $s>r$. Consequently, we see that
    \begin{equation}\label{lip integral}
        \varphi_r(y) = \int\limits_0^r\chi_{B(x,s)}(y)\dif\lb^1(s),
    \end{equation}
    because if we first assume that $r>|y-x|$, then $\chi_{B(x,s)}(y)=0$ when $s\in\left[0,|y-x|\right]$ and $\chi_{B(x,s)}(y)=1$ when $s\in\left(|y-x|,r\right]$. If $r\leq |y-x|$, then $\chi_{B(x,s)}(y)=0$ for all $s\in [0,r]$.
    
    Substituting $\varphi_r$ as test function in \eqref{curl as function equation} and multiplying across by $\frac{3}{\omega_2r^3}$ gives us that
    \begin{align*}
        \frac{3}{\omega_2r^3}\int_{\partial^*E\cap B(x,r)}\varphi_r\bm{q}^E\dif\h^2 &= \frac{3}{\omega_2r^3}\int_{E_r}\varphi_r\Curl\bm{F} - \frac{3}{\omega_2r^3}\int_{E_r}\bm{F}\times\nabla\varphi_r \dif\lb^3(y),
    \end{align*}
    where we have used \eqref{Er equiv} and the fact that $\varphi_r$ equals $0$ on $\Omega\setminus B(x,r)$.
    Considering the derivative of the Lipchitz function $\varphi_r$, which is defined $\lb^3$-a.e., we have that
    \begin{equation}\label{eq1}
        \frac{3}{\omega_2r^3}\int_{\partial^*E\cap B(x,r)}\varphi_r\bm{q}^E\dif\h^2 = \frac{3}{\omega_2r^3}\int_{E_r}\varphi_r\Curl\bm{F} - \frac{3}{\omega_2r^3}\int_{E_r}\bm{F}\times \frac{x-y}{|x-y|}\, \dif\lb^3(y)
    \end{equation}
    Next, \eqref{lip integral} gives us that 
    \begin{align*}
        \int\limits_{\rb E\cap B(x,r)}\varphi_r(y)\bm{q}^E(y)\dif\h^2(y) &= \int\limits_{\rb E\cap B(x,r)} \int\limits_0^r\chi_{B(x,s)}(y)\dif\lb^1(s) \bm{q}^E(y)\dif\h^2(y)\\
        &= \int\limits_0^r\int\limits_{\rb E\cap B(x,s)}\bm{q}^E(y) \dif\h^2(y)\dif\lb^1(s)
    \end{align*}
    which is the left-hand side term in \eqref{eq1} up to a multiplicative factor. Introducing that multiplicative factor of $\frac{3}{\omega_2 r^3}$, we see that
    \begin{equation}\label{FToC}
        \begin{aligned}
        \lim\limits_{r\to 0^+}\frac{3}{\omega_2r^3}\int_{\partial^*E\cap B(x,r)}\varphi_r\bm{q}^E\dif\h^2 &= \lim\limits_{r\to 0^+}\frac{3}{\omega_2r^3} \int\limits_0^r\int\limits_{\rb E\cap B(x,s)}\bm{q}^E(y) \dif\h^2(y)\dif\lb^1(s)\\
        &= \lim\limits_{r\to 0^+}\frac{1}{\omega_2r^2}\int\limits_{\rb E\cap B(x,r)}\bm{q}^E \dif\h^2 
    \end{aligned}
    \end{equation}
    where the last equality follows from an adaptation of L'Hôpital's rule to an absolutely continuous function, which we now prove.
    For $s\in[0,r]$, $f(s)\coloneqq\int\limits_{\rb E\cap B(x,s)}\bm{q}^E(y) \dif\h^2(y)$ defines a function in $L^1\left([0,r]\right)$. The fundamental theorem of calculus then implies that $r\mapsto \int_0^rf(s)\dif\lb^1(s)$ defines an absolutely continuous function.
    \eqref{first limit} gives us that $L\coloneqq\lim\limits_{s\to 0^+}\frac{f(s)}{s^2}$ exists and is finite.
    For any $\varepsilon>0$, there exists $r_0>0$, such that for all $s\in(0,r_0)$,
    \[
    (L-\varepsilon)s^2\leq f(s)\leq (L+\varepsilon)s^2.
    \]
    Integrating across the inequalities gives us that for all $r\in(0,r_0)$,
    \[
    (L-\varepsilon)\frac{r^3}{3}\leq \int_0^rf(s)\dif\lb^1(s)\leq (L+\varepsilon)\frac{r^3}{3}
    \]
    which simplifies to 
    \[
    (L-\varepsilon)\leq \frac{3}{r^3}\int_0^rf(s)\dif\lb^1(s)\leq (L+\varepsilon).
    \]
    Multiplying across these inequalities by $\frac{1}{\omega_2}$ and letting $r\to 0^+$ gives the final equality in \eqref{FToC}.
    Combining this with \eqref{first limit} gives us that the left hand side of \eqref{eq1} converges to $\bm{q}^E(x)$.
    Since for all $y\in E_r$, we have $(r-|y-x|)\in[0,r]$, we have
    \[
    \left| \int\limits_{E_r}(r-|y-x|)\dif\,(\Curl\F)\right| \leq r|\Curl\bm{F}|(E_r)
    \]
    which, combined with \eqref{second limit} gives us that the first term on the right hand side of \eqref{eq1} vanishes when the limit as $r\to 0^+$ is taken.
    Thus, we have obtained \eqref{trace is a function intermediary formula}.

    Recall that from the rectifiability of $\rb E$, we have \eqref{rectifiability}, which says that for any compact $K\subset\Omega$,
    \[
    \lim\limits_{r\to 0^+}|[E_{x,r}\Delta \{y\cdot\bm{\nu}_E(x)>0\}]\cap K|=0.
    \]
    In particular, letting $K=\overline{B(0,1)}$ and considering that $|\partial B(0,1)|=0$, this means that for $\h^2$-a.e. $x\in\rb E$,
    \begin{equation}\label{Federer tangential property}
        \frac{|\left[E\Delta\{(y-x)\cdot\bm{\nu}_E(x)>0\}\right]\cap B(x,r)|}{\omega_3 r^3}\to 0,
    \end{equation}
    as $r\to 0^+$. Next, we combine this result with \eqref{trace is a function intermediary formula}.
    We may replace the set that we integrate over in \eqref{trace is a function intermediary formula} with $B\left(x,\bm{\nu}_E(x\right), r)$, since\newline $B\left(x,\bm{\nu}_E(x\right), r)\Delta [E\cap B(x, r)] = [E\Delta \{(y-x)\cdot\bm{\nu}_E(x)>0\}]\cap B(x, r)$ and
    \begin{align*}
        \Bigg|\lim\limits_{r\to 0^+}\frac{3}{\omega_2r^3}&\int\limits_{ [E\Delta \{(y-x)\cdot\bm{\nu}_E(x)>0\}]\cap B(x, r)}\bm{F}(y)\times\frac{y-x}{|y-x|}\dif\lb^3(y) \Bigg|\\ 
        &\leq \frac{3||\F||_{L^\infty(\Omega)}}{\omega_3r^3} \left|  [E\Delta \{(y-x)\cdot\bm{\nu}_E(x)>0\}]\cap B(x, r) \right|
        \to 0, \qquad\text{as $r\to 0^+$}.
    \end{align*}
    Thus, we have \eqref{trace is a function expression}.
\end{proof}

We remark that while not actually required for the final result of this paper, the pointwise defined formula for the interior trace provided in Theorem \ref{Second trace theorem} provides some intuition for what the interior trace is. 
It is related to the idea of {\it blow-ups} as used by De Giorgi and is discussed in Chapter 10 of \cite{Maggi}. From the definition of Radon-Nikodym derivatives and the fact that $\bm{q}_2=\mF_e\times\n_E-\mF_i\times\n_E$, we can see that $\mF_e\times\n_E$ has a similar pointwise defined formula, the only difference being that we integrate over the other half of the ball and multiply by a minus sign. This construction of tangential traces via approximating functions is comparable to that of \cite{GG09} and \cite{One-sided-approx} where the {\it normal trace} is obtained via approximations of sets rather than functions. Using this idea of blow-ups, we can see why $\mF_i\times\n_E$ and $\mF_e\times\n_E$ depend on the measure-theoretic interior and exterior, respectively. However, in the proof of Theorem \ref{maintheorem}, we will show this dependence without explicitly utilizing the blow-up idea. 

\section{The tangential property}\label{tangential property section}
 Since the functions $\bm{q}^E$ and $\mF_i\times\n_E$ from $L^{\infty}(\rb E;\R^3,\h^2)$ satisfy the same Gauss-Green formulae \eqref{curl as function equation} and \eqref{uno}, respectively, for every test function $\varphi\in\Lip_c(\Omega)$, it follows from the fundamental lemma of the calculus of variations that they are the same function $\h^2$-a.e on $\rb E$. That is, 
    \[
    (\mF_i\times\n_E) (x) = \bm{q}^E (x), \quad \textnormal{ for } \h^{2}\textnormal{-a.e. } x \in \partial^* E.
    \]
In Section \ref{product method section} and Section \ref{functional method section}, we showed the existence of this function and its corresponding Gauss-Green formula using two different methods. Now, we will take ideas from both methods and use them to prove that both $\mF_i\times\n_E$ and $\mF_e\times\n_E$ are tangential to $\rb E$, $\h^2$-a.e. This is what we will show in the proof of Theorem \ref{maintheorem}, but in its statement we will also collect the results of the trace theorem (Theorem \ref{first trace theorem}) proved in Section \ref{product method section}. We do this so that we may clearly see that Theorem \ref{maintheorem} is a generalization of Theorem 3.9 of \cite{CurlMeasure25}, which makes the stronger assumption that $E$ is open with Lipschitz boundary. The tangential property is fundamental in that paper, however, its proof relies upon properties of Lipschitz boundaries which are not available under our assumptions. It is for this reason that we must take a new approach in the proof of the following theorem.
\begin{theorem}
\label{maintheorem}
    Let $\bm{F}\in\mathcal{CM}^\infty(\Omega)$, $E\Subset\Omega\subset\R^3$ a set of finite perimeter.
    Then there exists $\mF_i\times\n_E,\mF_e\times\n_E\in L^\infty(\rb E;\R^3)$ such that, for all $\varphi\in\Lip_c(\Omega)$,
    \begin{equation*}
       \int_{\partial^*E}\varphi(\mF_i\times\n_E)\dif\h^2 = \int_{E^1}\varphi\dif\,(\Curl\F)-\int_{E}\bm{F}\times\nabla\varphi \dif\lb^3
    \end{equation*}
    and
    \begin{equation*}
       \int_{\partial^*E}\varphi(\mF_e\times\n_E)\dif\h^2 = \int_{E^1\cup\rb E}\varphi\dif\,(\Curl\F)-\int_{E}\bm{F}\times\nabla\varphi \dif\lb^3.
    \end{equation*}
    Moreover, we can choose representatives for these functions in $L^\infty(\rb E;\R^3)$ such that for $\h^2$-a.e $x\in\rb E$, $(\mF_i\times\n_E)(x),(\mF_e\times\n_E)(x) \in T_{x}E$.
\end{theorem}
\begin{proof}
    The hypotheses are the same as in Theorem \ref{Second trace theorem}, so we  have, for all $\varphi\in\Lip_c(\Omega)$,
    \begin{equation}
      \label{one} \int_{\partial^*E}\varphi\bm{q}^E\dif\h^2 = \int_{E^1}\varphi\dif\,(\Curl\bm{F})-\int_{E}\bm{F}\times\nabla\varphi \dif\lb^3.
    \end{equation}
    We consider the vectorial version of \eqref{one}. That is, for all $\varphi\in\Lip_c(\Omega;\R^3)$, 
    \begin{equation}
     \int_{\partial^*E}\varphi\cdot\bm{q}^E\dif\h^2 = \int_{E^1}\varphi\cdot \dif\, (\Curl\bm{F})-\int_{E}(\Curl\varphi)\cdot\bm{F}\dif\lb^3.
    \end{equation}
    From the definition of $\bm{q}_0\in L^\infty(\rb E;\R^3,\h^2)\subseteq L^1(\rb E;\R^3,\h^2)$, as $\varepsilon\to 0^+$,
    \begin{equation}
    \label{q_0 weak}
        \int_{\partial^*E}\bm{F}_\varepsilon\times\bm{\nu}_E \, g \dif\h^2 \longrightarrow \int_{\partial^*E}\bm{q}_0 \, g \dif\h^2,
    \end{equation}
    for any $g\in L^1(\rb E,\h^2)$. We can also write a vectorial form of the previous limit by applying \eqref{q_0 weak} to each component of the vector field $\bm{g}\in L^1(\partial^*E;\R^3,\h^2)$. Therefore, for any such ${\bm g}$, 
    \begin{equation}\label{q_0 vectorial}
        \int_{\partial^*E}\left(\bm{F}_\varepsilon\times\bm{\nu}_E\right) \, \cdot\bm{g} \dif\h^2 \longrightarrow \int_{\partial^*E}\bm{q}_0 \, \cdot\bm{g} \dif\h^2,
    \end{equation}
    as $\varepsilon\to 0^+$.
    
    Since $\bm{\nu}_E:\rb E\to S^2$ is Borel function and $\h^2(\rb E)<\infty$, it follows that $\bm{\nu}_E\in L^1(\rb E;\R^3,\h^2)$. Therefore, from \eqref{q_0 vectorial} we obtain, for any $\phi\in C_c(\Omega)$, 
    \[
    \int_{\partial^*E}\underbrace{\left(\bm{F}_\varepsilon\times\bm{\nu}_E\right) \, \cdot (\phi\bm{\nu}_E)}_{=\,0} \dif\h^2 \longrightarrow \int_{\partial^*E}\bm{q}_0 \, \cdot (\phi\bm{\nu}_E) \dif\h^2,
    \]
    as $\varepsilon\to 0^+$. We have shown that
    \[
    \int_{\partial^*E} \left(\bm{q}_0 \, \cdot \bm{\nu}_E\right)\,\phi\, \dif\h^2 =0, \quad \textnormal{ for any } \phi\in C_c(\Omega),
    \]
    or equivalently,
    \begin{equation}
    \label{ort}
    \int_{\Omega}\phi \left(\bm{q}_0 \, \cdot \bm{\nu}_E\right)\,\, \dif \,(\h^2\rtangle \rb E) =0, \quad \textnormal{ for any } \phi\in C_c(\Omega).
    \end{equation}
    Now, from \eqref{ort} and the fundamental lemma of the calculus of variations, we conclude that, for $\h^2$-a.e. $x\in\rb E$,
    \[
     \bm{q}_0 (x) \cdot \bm{\nu}_E (x)  = 0
    \]
    that is, for $\h^2$-a.e. $x\in\rb E$, $\bm{q}_0(x)\in T_{x}E$.
    
    Again, we remark that 
    \[
    \bm{q}^E (x) = (\mF_i\times\n_E) (x), \quad \textnormal{ for } \h^{2}\textnormal{-a.e. } x \in \partial^* E
    \]
    and since the statements that we wish to prove need only hold $\h^2$-a.e., we may use $\mF_i\times\n_E$ and $\bm{q}^E$ interchangeably.
    From \eqref{f3}, the definition of $\bm{q}^E$ in \eqref{q0 defn}, and using that $\h^2\rtangle \rb E = |\nabla\chi_E|$, we deduce that
    \begin{align*}
        \bm{q}_0 &= \bm{q}^E +\frac{1}{2}\bm{q}_2\\
        &= \mF_i\times\n_E+\frac{1}{2}\left(\mF_e\times\n_E-\mF_i\times\n_E\right)\\
        &= \frac{1}{2}\left(\mF_i\times\n_E +\mF_e\times\n_E\right).
    \end{align*}
    Now, consider $\widetilde{\F}\coloneqq\chi_{E}\F$. The product rule, Theorem \ref{ProductRule}
    yields $\widetilde{\F}\in\mathcal{CM}^\infty(\Omega)$. Let $\widetilde{\mF_i\times\n_E}$ and $\widetilde{\mF_e\times\n_E}$ be the interior and exterior tangential traces of $\widetilde{\F}$, respectively. We consider the integrals in the right hand side of \eqref{ladef} over $E^1$. The second integral agrees for both $\F$ and $\widetilde{\F}$ since $\lb^3(E \Delta E_1) =0$. Also, the first integral the takes the same value for both $\F$ and $\widetilde{\F}$ since, by \eqref{f1}, $\Curl \widetilde{\F}=\chi_{E^1}\Curl \F-\mF_i\times\n_E|\nabla\chi_E|$ and $|\nabla\chi_E| (E^1)= \h^2(\partial^*E \cap E^1)=0$. Thus,
    \[
    \mF_i\times\bm{\nu}_E =\widetilde{\mF_i\times\bm{\nu}_E}.
    \]
  We now consider the integrals in the right hand side of \eqref{ladef2} for $\widetilde{\F}$ over $E^0$. Using again that $\Curl\widetilde{\F}=\chi_{E^1}\Curl \F-\mF_i\times\n_E|\nabla\chi_E|$ it follows that the first integral is zero since the measure $\Curl \widetilde{\F}$ is concentrated on $E^1 \cap \partial^* E$. The second integral is also zero since $\lb^3(E \Delta E_1) =0$ and hence $\int_{E^0} \widetilde{\F} \times \nabla \varphi \dif \, x = \int_{E^0 \cap E} \F \times \nabla \varphi \dif x= \int_{E^0 \cap E^1} \F \times \nabla \varphi \dif x=\bm{0}$. We have shown that $\h^2$-a.e.
    \[
    \widetilde{\mF_e\times\n_E}=\bm{0}.
    \]
    Above, we saw that $\bm{q}_0=\frac{1}{2}\left(\mF_i\times\n_E +\mF_e\times\n_E\right)$ is tangential to $\rb E$, $\h^2$-a.e. In the same way, we have that $\widetilde{\bm{q}_0}\coloneqq \frac{1}{2}\left(\widetilde{\mF_i\times\n_E} +\widetilde{\mF_e\times\n_E}\right)$ is tangential to $\rb E$, $\h^2$-a.e.
    So, for $\h^2$-a.e. $x\in\rb E$, we have that
    \[
    \frac{1}{2}\left(\left(\widetilde{\mF_i\times\bm{\nu}_E}\right)(x)\, +\bm{0}\right)\in T_xE
    \]
    and since $\mF_i\times\bm{\nu}_E =\widetilde{\mF_i\times\bm{\nu}_E}$, we have that $\mF_i\times\bm{\nu}_E\in T_xE$ for $\h^2$-a.e. $x\in\rb E$. Of course, since $\bm{q}_0 = \frac{1}{2}\left(\mF_i\times\n_E +\mF_e\times\n_E\right)$ is tangential to $\rb E$, $\h^2$-a.e., we also have that $\left(\mF_e\times\n_E\right)(x)\in T_xE$ for $\h^2$-a.e. $x\in\rb E$.
\end{proof}

Returning to the context of vortex sheets that was discussed in Section \ref{tangential traces section}, we note that this result confirms that $\F\times\n_E$ coming from either side of the vortex sheet $\rb E$ is indeed tangential to $\rb E$, $\h^2$-a.e., giving credence to this mathematical modeling of fluid mechanics. 
We also note that after rotation of $90^\circ$ in the tangent plane, the function $\bm{q}_0=\frac{1}{2}\left(\mF_i\times\n_E +\mF_e\times\n_E\right)$ corresponds to what is known as the {\it mean tangential velocity along the vortex sheet}. As seen in \cite{CurlMeasure25}, this quantity is used when deriving the Birkhoff-Rott equation which describes the evolution of vortex sheets over time.

We now return to our motivating example from Section \ref{cm fields section} and apply our results.
\begin{continuation}{example0}\label{example0 cont}
    Recall that in this example we are considering a unit cube $C$ in $\R^3$ with one face, $S$, that lies in a plane $P$ and $\F(x,y,z)=\sin\left(\frac{1}{x+y+z}\right)(1,1,1)$.
    We first show that, $\F\in\mathcal{CM}^\infty(\R^3)$. Clearly, $\F\in L^\infty(\R^3)$. Considering the action of the distributional curl of $\F$ on $\varphi\in C^\infty_c(\R^3)$, we see that
\begin{align*}
    \langle\Curl\F,\varphi\rangle &= \int_{\R^3}\F\times\nabla\varphi\dif\lb^3 
    = \lim\limits_{\varepsilon\to 0^+}\int_{(\R^3\setminus P)^\varepsilon} \F\times\nabla\varphi\dif\lb^3\\
    &=\lim\limits_{\varepsilon\to 0^+} \int_{(\R^3\setminus P)^\varepsilon} \varphi\Curl\F \dif\lb^3 - \int_{\partial\left((\R^3\setminus P)^\varepsilon\right)} \varphi \F\times\n_{\partial\left((\R^3\setminus P)^\varepsilon\right)}\dif\h^2 = \bm{0}    
\end{align*}
where the second equality follows by considering the dominated convergence theorem and the third equality follows by the classical Gauss-Green formula \eqref{mainformula2}. The fourth equality follows since $\Curl\F$ and $\F\times\n_{\partial\left((\R^3\setminus P)^\varepsilon\right)}$ both evaluate to $\bm{0}$ on $\R^3\setminus P$ and $\partial\left((\R^3\setminus P)^\varepsilon\right)$, respectively.

Theorem \ref{maintheorem} gives us the existence of Gauss-Green formulae as well as the existence of interior tangential traces that are tangential to $\rb C$.
We now use the pointwise defined formula for the interior tangential trace in Theorem \ref{Second trace theorem} to see this explicitly. Consider $a=(a_1,a_2,a_3)\in\rb C$. That is, $a$ lies on one of the six faces of $C$, but is not on an edge or corner. Note that the limit in formula \eqref{trace is a function expression} exists and is finite for $\h^2$-a.e. point on $\rb C$.
However, our calculations will show that, in this example, the limit exists and is finite for {\it all} points in $\partial^* C$.
Considering this, formula \eqref{trace is a function expression} gives us that
\begin{equation}\label{trace explicit}
    \left(\mF_i\times\n_{C}\right)(a)= \lim\limits_{r\to 0^+}\frac{3}{\omega_2r^3}\int\limits_{B(a, \n_C, r)} \bm{F}(u)\times\frac{u-a}{|u-a|}\dif\lb^3(u).
\end{equation}
To simplify our notation we let $H_r(a)\coloneqq B(a, \n_C, r)$ and for $r>0$, we let
\begin{equation}\label{I(r)}
    I(r)\coloneqq \int\limits_{H_r(a)} \bm{F}(u)\times\frac{u-a}{|u-a|}\dif\lb^3(u).
\end{equation}
Thus,
\[
\left(\mF_i\times\n_C\right)(a)=\lim\limits_{r\to 0^+}\frac{3}{\omega_2r^3}I(r).
\]
Note that for $r>0$ sufficiently small, $H_r(a)$ is a half-ball with a base centered at $a$ that lies inside $C$. From now on, we shall assume that any $r>0$ considered is sufficiently small.
We first consider the case where $a$ lies on the face that intersects $P$, that is $a\in S$. For any $r>0$, we have
\[
I(r)= \int_{H_r(a)}\left(\sin\left(\frac{1}{x+y+z}\right)\right)(1,1,1)\times\frac{(x,y,z)-a}{|(x,y,z)-a|}\dif\lb^3
\]
We make a coordinate change that translates $a$ to the origin and then rotates $P$ while fixing the origin. That is, we let
\begin{equation}\label{cov}
\begin{bmatrix}
    x\\ y\\ z
\end{bmatrix}
=
\begin{bmatrix}
    a_1\\ a_2\\ a_3
\end{bmatrix}
+
\begin{bmatrix}
    1/\sqrt{2} & 1/\sqrt{6} & 1/\sqrt{3}\\
    -1/\sqrt{2} & 1/\sqrt{6} & 1/\sqrt{3}\\
    0 & -2/\sqrt{6} & 1/\sqrt{3}
\end{bmatrix}
\begin{bmatrix}
    \xi\\ \eta\\ \zeta
\end{bmatrix}
\end{equation}
and letting $R$ denote the rotation matrix in \eqref{cov}, we then see that
\[
I(r)= R\left(\int_{H_r(\bm{0})} \sin\left(\frac{1}{\sqrt{3}\zeta}\right)\left(R^{-1}(1,1,1)\right)\times\frac{\left(\xi,\eta,\zeta\right)}{|\left(\xi,\eta,\zeta\right)|} \dif\lb^3\right),
\]
where $H_r(\bm{0})$ is the half-ball centered at the origin that lies in $\{(x,y,z) : z>0\}$. Simplifying, we see that
\[
I(r)= R\left(\int_{H_r(\bm{0})} \frac{\sqrt{3}\sin\left(\frac{1}{\sqrt{3}\zeta}\right)}{\sqrt{\xi^2+\eta^2+\zeta^2}}\left(-\sqrt{3}\eta,\sqrt{3}\xi,0\right) \dif\lb^3\right).
\]
Notice that the first component of the vector being integrated is odd with respect to $\eta$ and the second component is odd with respect to $\xi$. Since this iterated integral involves integrating with respect to $\xi$ and $\eta$ over symmetric regions about the origin, the first two components evaluate to zero. The third component is already zero before integration.
Therefore, for all $a\in S$, $\left(\mF_i\times\n_C\right)(a)=\bm{0}$.
If $a$ lies on one of the faces that is not intersecting $P$, then, as we expect,
\begin{equation}\label{cts trace}
    \left(\mF_i\times\n_{C}\right)(a)=\F(a)\times\n_C(a),
\end{equation}
since $\F$ is continuous in $H_r(a)$ for $r>0$. Explicitly, we note that
\begin{align*}
    \frac{3}{\omega_2r^3}&\int\limits_{H_r(a)} \bm{F}(u)\times\frac{u-a}{|u-a|}\dif\lb^3(u)\\ &= \frac{3}{\omega_2r^3}\F(a)\times\int\limits_{H_r(a)}\frac{u-a}{|u-a|}\dif\lb^3(u) +\frac{3}{\omega_2r^3}\int\limits_{H_r(a)} \left(\F(u)-\F(a)\right)\times\frac{u-a}{|u-a|}\dif\lb^3(u).
\end{align*}
A simple calculation shows that the first integral on the right hand side of the previous equation equals $\frac{\omega_2 r^3}{3}\n_C(a)$. After noting that for all $r>0$, $|H_r(a)|>\frac{1}{3}|B(a,r)|$, we may apply the Lebesgue differentiation theorem (Theorem 3.21 in \cite{Folland}) to the other integral and note that for continuous functions, every point is a Lebesgue point. Thus, taking $r\to 0^+$ in the previous equation gives us \eqref{cts trace}.
Therefore, we see that on $\rb C$, $\left(\mF_i\times\n_C\right)\cdot\n_C=0$, which demonstrates the tangential property that we proved in Theorem \ref{maintheorem}. 

We next demonstrate that $\mF_i\times\n_C \in L^{\infty}(\partial^* C)$ satisfies the Gauss-Green formula also obtained in \ref{maintheorem}, that is
\begin{equation}\label{GG eg}
    \int_{\rb C}(\mF_i\times\n_C)(y) \dif\h^2 (y) =\int_C\Curl\F\dif\lb^3.
\end{equation}
We consider the integral of $\mF_i\times\n_C$ on the six individual faces of $C$. On $S$, the face lying in $P$, we have already seen that $\mF_i\times\n_C=\bm{0}$, which implies that $\int_{S}\mF_i\times\n_C \dif\h^2 =\bm{0}$. On the opposite face, denoted by $S'$, \eqref{cts trace} gives us that $\mF_i\times\n_C=\F\times\left(-(1,1,1)\right)=\bm{0}$. Hence, $\int_{S'}\mF_i\times\n_C \dif\h^2 =\bm{0}$. Of the remaining four faces, we pick two that are opposite to each other and denote them by $A_1$ and $A_2$. $\n_C$ is a constant $\bm{b}$ that only differs by a sign on $A_1$ and $A_2$ and $\F$ is constant along all planes parallel to $P$. This implies that
\[
\int_{A_1}\mF_i\times\n_C \dif\h^2 = \int_{A_1}\F\times\bm{b} \dif\h^2 = \int_{A_2}\F\times\left(-\bm{b}\right) \dif\h^2 = -\int_{A_2}\mF_i\times\n_C \dif\h^2.
\]
With this, we have cancellation of integrals over opposite sides of $C$ and thus, $\int_{\rb C}(\mF_i\times\n_C)(y) \dif\h^2(y)=\bm{0}$.
Recalling that $\int_C\Curl\F\dif\lb^3=\bm{0}$, we see that \eqref{GG eg} holds.

This example was chosen for the simple calculations involved, while also demonstrating the application of Theorem \ref{maintheorem} to a nontrivial scenario. On $S$, where $\F$ is undefined, the interior tangential trace, $\mF_i\times\n_C$, reduces to $\bm{0}$. We could replace $\F$ in this example with $\F+\nabla f$, where $f\in C^2(\R^3)$ is chosen such that $\nabla f\in L^\infty(\R^3;\R^3)$ and $\mF_i\times\n_C$ is nonzero on $S$. If we make this replacement, similar calculations to those above show that, on $S$, $\mF_i\times\n_C = \nabla f\times\n_C$. 
While this example allows for explicit computations, Theorem \ref{maintheorem} applies to any essentially bounded curl-measure field and domain of finite perimeter.
\end{continuation}

\bigskip
\medskip
\medskip


\end{document}